\documentclass[11pt]{article}

\usepackage[letterpaper, hmargin=1in, top=1in, bottom=1.1in, footskip=0.6in]{geometry}

\usepackage{titlesec}
\titleformat{\subsection}[runin]{\normalfont\bfseries}{\thesubsection.}{.5em}{}[.]\titlespacing{\subsection}{0pt}{2ex plus .1ex minus .2ex}{.8em}
\titleformat{\subsubsection}[runin]{\normalfont\itshape}{\thesubsubsection.}{.3em}{}[.]\titlespacing{\subsubsection}{0pt}{1ex plus .1ex minus .2ex}{.5em}
\titleformat{\paragraph}[runin]{\normalfont\itshape}{\theparagraph.}{.3em}{}[.]\titlespacing{\paragraph}{0pt}{1ex plus .1ex minus .2ex}{.5em}

\usepackage{amsmath}
\usepackage{amssymb}
\usepackage{amsfonts}
\usepackage{latexsym}
\usepackage{amsthm}
\usepackage{amsxtra}
\usepackage{amscd}
\usepackage{bbm}
\usepackage{mathrsfs}
\usepackage{bm}
\usepackage{tensor}

\usepackage{graphicx, color}

\definecolor{darkred}{rgb}{0.9,0,0.3}
\definecolor{darkblue}{rgb}{0,0.3,0.9}

\newcommand{\nc}{\normalcolor}

\definecolor{vdarkred}{rgb}{0.7,0,0.2}
\definecolor{vdarkblue}{rgb}{0,0.2,0.7}
\usepackage[pdftex, colorlinks, linkcolor=vdarkblue,citecolor=vdarkred]{hyperref}

\usepackage[nottoc,notlof,notlot]{tocbibind}
\usepackage{cite}

\numberwithin{equation}{section}
\numberwithin{figure}{section}

\theoremstyle{plain} 
\newtheorem{theorem}{Theorem}[section]
\newtheorem*{theorem*}{Theorem}
\newtheorem{lemma}[theorem]{Lemma}
\newtheorem*{lemma*}{Lemma}

\newtheorem*{corollary*}{Corollary}
\newtheorem{proposition}[theorem]{Proposition}
\newtheorem*{proposition*}{Proposition}

\newtheorem*{conjecture*}{Conjecture}

\theoremstyle{definition} 
\newtheorem{definition}[theorem]{Definition}
\newtheorem*{definition*}{Definition}
\newtheorem{example}[theorem]{Example}
\newtheorem*{example*}{Example}
\newtheorem{remark}[theorem]{Remark}
\newtheorem*{remark*}{Remark}

\newtheorem*{assumption*}{Assumption}

\renewcommand{\b}[1]{\boldsymbol{\mathrm{#1}}} 
\newcommand{\bb}{\mathbb} 
\renewcommand{\cal}{\mathcal}

\newcommand{\ul}[1]{\underline{#1} \!\,} 
\newcommand{\ol}[1]{\overline{#1} \!\,} 

\newcommand{\txt}[1]{\text{\rm{#1}}}

\newcommand{\E}{\mathbb{E}}
\newcommand{\R}{\mathbb{R}}
\newcommand{\C}{\mathbb{C}}
\newcommand{\N}{\mathbb{N}}

\newcommand{\e}{\mathrm{e}}

\newcommand{\ii}{\mathrm{i}}
\newcommand{\dd}{\mathrm{d}}

\newcommand*{\deq}{\mathrel{\vcenter{\baselineskip0.65ex \lineskiplimit0pt \hbox{.}\hbox{.}}}=}
\newcommand*{\eqd}{=\mathrel{\vcenter{\baselineskip0.65ex \lineskiplimit0pt \hbox{.}\hbox{.}}}}

\renewcommand{\leq}{\leqslant}
\renewcommand{\le}{\leqslant}
\renewcommand{\geq}{\geqslant}
\renewcommand{\ge}{\geqslant}
\renewcommand{\epsilon}{\varepsilon}

\newcommand{\pb}[1]{\bigl({#1}\bigr)}

\DeclareMathOperator{\tr}{Tr}

\DeclareMathOperator{\im}{Im}

\newcommand{\bC}{ {\mathbb C} }
\newcommand{\bN}{ {\mathbb N} }
\newcommand{\bE}{ {\mathbb E} }

\newcommand{\bP}{ {\mathbb P} }
\newcommand{\bR}{ {\mathbb R} }

\newcommand*{\rom}[1]{\expandafter\@slowromancap\romannumeral #1@}

\title{Bulk eigenvalue fluctuations of sparse random matrices}
\author{Yukun He\footnote{University of Z\"{u}rich, Institute of Mathematics. Email: {\tt yukun.he@math.uzh.ch}.} }

\begin{document}
\maketitle

\begin{abstract}
	We consider a class of sparse random matrices, which includes the adjacency matrix of Erd\H{o}s-R\'enyi graphs $\mathcal G(N,p)$ for $p \in [N^{\varepsilon-1},N^{-\varepsilon}]$. We identify the joint limiting distributions of the eigenvalues away from 0 and the spectral edges. Our result indicates that unlike Wigner matrices, the eigenvalues of sparse matrices satisfy central limit theorems with normalization $N\sqrt{p}$. In addition, the eigenvalues fluctuate simultaneously: the correlation of two eigenvalues of the same/different sign is asymptotically 1/-1. We also prove CLTs for the eigenvalue counting function and trace of the resolvent at mesoscopic scales.
\end{abstract}

\section{Introduction and statements of results} \label{sec:2}

Let $\cal A$ be the adjacency matrix of a sparse Erd\H{o}s-R\'enyi graph $\cal G(N,p)$. That is, $\cal A$ is a symmetric $N\times N$ matrix with independent upper triangular entries satisfying
\begin{equation*}
{\cal A}_{ij}=\begin{cases}
1 & \txt{with probability } p
\\
0 & \txt{with probability } 1-p\,.
\end{cases}
\end{equation*} 
Note that each row and column of $\cal A$ has typically $Np$ nonzero entries, and we are interested in the case when $\cal A$ is sparse; more precisely, we set $p \in [N^{-1+\epsilon},N^{-\epsilon}]$ for some fixed $\epsilon>0$. It is convenient to introduce the normalized matrix
\begin{equation} \label{1.11}
A\deq \sqrt{\frac{1}{p(1-p)N}}\,\cal A
\end{equation}
so that the typical eigenvalue spacing of $A$ is of order $N^{-1}$. We also introduce the new variable
\[
q\deq \sqrt{Np}\,.
\]
In this paper, we consider random matrices of the following class; it is an easy exercise to check that $A$ defined in \eqref{1.11} in terms of $\cal G(N, p)$ satisfies the following conditions.

\begin{definition} [Sparse matrix] \label{def:sperse} Fix $\beta \in (0,1/2)$ and set $q \deq N^{\beta}$. A sparse matrix is a real symmetric $N\times N$ matrix $H=H^* \in \bb R^{N \times N}$ whose entries $H_{ij}$ satisfy the following conditions.
	\begin{enumerate}
		\item[(i)] The upper-triangular entries ($H_{ij}: 1 \leq i \leq j\leq N$) are independent.
		\item[(ii)] The off-diagonal entries ($H_{ij}:  i \ne j$) are identically distributed.
		\item[(iii)] We have $\bb E H_{ij}=0$ and $ \bb E H_{ij}^2=(1+O(\delta_{ij}))/N$ for all $i,j$.
		\item[(iv)] For any $k\geq 3$, we have
		$\bb E|H_{ij}|^k \leq C_k/(Nq^{k-2})$ for all $i,j$.
	\end{enumerate}
	We define the adjacency matrix $A$ by
	$$
	A = H + f \b e \b e^*\,,
	$$
	where $\b e \deq N^{-1/2}(1,1,\dots,1)^*$, and $f \geq 0$.
\end{definition}
A special case of the above model is the Wigner matrix. Recall that Wigner matrix is an $N\times N$ real symmetric matrix $W$ satisfying the assumptions (i) -- (iii) in Definition \ref{def:sperse}, and $\|W_{ij}\|_k\asymp\|W_{ij}\|_2$ for all $k\geq 3$. $W$ is the Gaussian Orthogonal Ensemble (GOE) if we further assume that $W_{ij}$ have Gaussian distributions. 

The celebrated Wigner-Dyson-Mehta (WDM) universality conjecture asserts that the local spectral properties of a random matrix do not depend on the explicit distribution of the matrix entries, and they are only determined by the symmetry class of the matrix. During the past decade, the universality conjecture for Wigner matrices has been established in a series of papers \cite{EPRSY,ESY4,ESY1,ESY3,TV1,TV2} in great generality. In particular, it has been shown that for a symmetric Wigner matrix, the averaged $n$-point correlation functions and distribution of a single eigenvalue gap coincide with those of the GOE.

The study of universality for sparse matrices was initiated in \cite{EKYY1, EKYY2}, where the authors proved local semicircle law on optimal scales, and established bulk universality for $q \geq N^{1/3}$. Later in \cite{HLY15}, the result was extended to all $q \geq N^{\epsilon}$. In particular, it was proved that for the eigenvalues $\lambda_1 \leq \dots \leq \lambda_N$ of $A$ and the eigenvalues $\mu_1\leq \dots \leq \mu_N$ of GOE, one has
\begin{equation} \label{gap}
\lim_{N\to \infty}\bb E [f(N\varrho(\gamma_i)(\lambda_i-\lambda_{i+1}))-f(N\varrho(\gamma_i)(\mu_i-\mu_{i+1}))] =0
\end{equation}
whenever $i \in [\epsilon N,(1-\epsilon)N]$. Here $f \in C^{\infty}_c(\bb R)$, $\varrho$ and $\gamma_i$ are the semicircle density and its $i$th $N$-quantile $\gamma_i$ respectively, i.e.
\[
\varrho(x)\deq \frac{1}{2\pi}\sqrt{(4-x^2)_{+}} \quad \mbox{and} \quad N\int_{-2}^{\gamma_i}\varrho(x) \dd x =i-1/2\,.
\]

Unlike the averaged n-point correlation functions and single eigenvalue gaps, the fluctuations of single eigenvalues are understood much later. The single eigenvalue fluctuation was first considered in \cite{G2005} for Gaussian Unitary Ensembles (GUE), where the author proved that 
\begin{equation} \label{1.22}
\frac{\mu_i-\gamma_i}{\sqrt{\frac{2\log N}{(4-\gamma_i^2)N^2}}} \overset{d}{\longrightarrow} \cal N(0,1)
\end{equation}
as $N\to \infty$, for all bulk eigenvalues $\mu_i$ of GUE.  In \cite{O2010}, the result was extended to GOE and a special class of Wigner matrices. Recently in \cite{BM18,LS18}, it was showed that \eqref{1.22} remains valid for all Wigner matrices.

In this paper, we study the single eigenvalue fluctuation of the sparse matrices. For the remaining of this paper we replace the assumption (iv) in Definition \ref{def:sperse} by
\begin{enumerate}
	\item[(iv)$'$] For any $k \geq 3$, we have $\bb E|H_{ij}|^k \asymp 1/(Nq^{k-2})$ for all $i,j$,
\end{enumerate}
so that $H$ and $A$ are strictly sparse.  Let us denote
\begin{equation} \label{zeta}
\zeta\deq \min\Big\{\frac{1}{2}-\beta,\beta\Big\}>0\,.
\end{equation}
We may now state our main result.

\begin{theorem}[Main result] \label{mainthm} Fix $\tau>0$. Let $\lambda_1 \leq \lambda_2 \leq \cdots \leq \lambda_N$ be the eigenvalues of $A$.  Set
	\[
	X_i\deq \frac{\lambda_i-\bb E \lambda_i}{\gamma_i \sqrt{ \frac{1}{2}\bb E H^4_{12}}}
	\] 
	for all $i \in \{1,2,...,N\}$. We denote the index set $\cal I\deq ([\tau N, N/2-N^{1-\zeta/17}] \cup [N/2+N^{1-\zeta/17},(1-\tau N)]) \cap \bb N$. Then for any fixed $k$ and $i_1,...,i_k \in \cal I$,
\begin{equation} \label{result1}
(X_{i_1},...,X_{i_k}) \overset{d}{\longrightarrow} \cal N_k({\bf 0}, \cal J)\,,
\end{equation}
where $\cal J\in \bR^{k\times k}$ is the matrix of ones, i.e. $\cal J_{ij}= 1$ for all $i,j\in\{1,2,...,k\}$. 

For $\lim_{N\to \infty} i/N=1/2$, we have
\[
\frac{\lambda_i-\bb E \lambda_i}{\sqrt{\bb E H_{12}^4}} \overset{d}{\longrightarrow} 0\,.
\]	
\end{theorem}

By assumption (iv)$'$, we have $\gamma_i \sqrt{ \frac{1}{2}\bb E H^4_{12}} \asymp |\gamma_i|N^{-1/2-\beta}$. Thus \eqref{result1} implies that, when $\gamma_i$ is away from 0, the corresponding $\lambda_i$ fluctuates on a much larger scale than the bulk eigenvalues of a Wigner matrix. Since the limiting covariance matrix $\cal J$ is the matrix of ones, we see that all the eigenvalues fluctuate simultaneously. Note the phenomenon of the co-existence of Theorem \ref{mainthm} and the gap universality \eqref{gap}: although the eigenvalues of $A$ fluctuate on large scales, the fluctuations of consecutive eigenvalues are almost identical, and hence the fluctuations make little impact on the gap distribution.

We also remark on the fluctuation near the edge. For $q \gg N^{1/6}$, the extreme eigenvalues of $A$ are known to exhibit Tracy-Widom fluctuations \cite{LS1,EKYY2}.  When $N^{1/9} \ll q \ll N^{1/6}$, it was proved in \cite{HLY} that 
\begin{equation} \label{1.6}
\frac{\lambda_{N-1}-(2+1/q^2-5/(4q^4))}{\sqrt{ 2\bb E H^4_{12}}} \overset{d}{\longrightarrow} \cal N(0,1)\,.
\end{equation}
Note that
\[
X_{N-1}=\frac{\lambda_{N-1}-\bb E \lambda_{N-1}}{\sqrt{ 2\bb E H^4_{12}}}(1+O(N^{-2/3}))\,,
\]
thus for $N^{1/9} \ll q \ll N^{1/6}$, the bulk fluctuation \eqref{result1} exhibits exactly the same behavior as the edge fluctuation \eqref{1.6}. In fact, in both cases the fluctuations come from the sparsity of $A$.  We believe that the source of the edge fluctuation remains the same for small $q$, and Theorem \ref{mainthm} can be extended to the edge for all $N^{\epsilon}\leq q \ll N^{1/6}$.

We also have the following Central Limit Theorem for the eigenvalue counting function of $A$.
\begin{theorem} \label{cor1.4}
Fix $\tau>0$.	Let $ \Sigma(E)\deq |\{i:\lambda_i \leq E\}|$ denote the eigenvalue counting function of $A$. For $E \in [-2+\tau,-N^{-\zeta/17}]\cup[N^{-\zeta/17},2-\tau]$, we have
	\[
	\frac{\Sigma(E)- \bb E \Sigma(E) }{\sigma(E)} \overset{d}{\longrightarrow} \cal N\big(0, 1\big)\,,
	\]
	where
\[
\sigma(E) \deq E\sqrt{4-E^2}\bigg(\frac{\bb E H_{12}^4}{8\pi^2}\bigg)^{1/2} N\,.
\]
For $\lim_{N \to \infty} E=0$, we have
\[
\frac{\sqrt{N}}{q}\big(\Sigma(E)- \bb E \Sigma(E) \big) \overset{d}{\longrightarrow} 0\,.
\]
\end{theorem}

 Let $F$ be a smooth test function independent of $N$. Recall that for a Wigner matrix $W$, the macroscopic linear statistic $\tr F(W)$ fluctuates on the scale 1 (see \cite{LP}), while \cite{LS18} shows $\Sigma_W(E)$ fluctuates on the scale $\sqrt{\log N} \gg 1$. This is due to the fact that as the derivative of the test function becomes more singular, the leading contribution of the fluctuation will start to come from the fluctuations of individual eigenvalues, which are much larger than the averaging fluctuation from linear statistics.

 Our observation is that for a sparse matrix, $\Sigma(E)$ should fluctuate on the same scale as $\tr F(H)$. From \cite{ST12} we know that $\tr F(H)$ fluctuates on the scale $\sqrt{N}/q\gg \sqrt{\log N}$. The source of this is the fourth moment assumption $\bb E (\sqrt{N}H_{12})^4\asymp N/q^2 \gg 1$, which gives rise to large, but simultaneous fluctuations for all eigenvalues. When we switch from continuous to a jump test function, the result remains the same, as the source of the fluctuation is unchanged.

To study $\lambda_i$ and $\Sigma(E)$, the main step is obtaining good estimates for linear statistics of Green functions at small scales. Let us define the spectral domain
\[
{\bf D}_{\tau} \deq \{E+\ii \eta: |4-E^2|+\eta\geq \tau,|E|\leq 4, N^{-1+\tau}\leq \eta \leq 4\}\,.
\]
We denote the resolvent of $H$ by $G(z)\deq (H-z)^{-1}$, where $\im z \ne 0$.
The key step of our proof is a result on centered moments of mesoscopic linear statistics of the Green functions (see Proposition \ref{propgreen} below), which in particular implies the optimal estimate
\[
\frac{1}{N}\tr G(z) -\frac{1}{N}\bb E \tr G(z) \prec \frac{1}{N\eta}+\frac{1}{\sqrt{N}q}
\]
for all $z=E+\ii \eta \in {\bf D}_{\tau}$. Here $``\prec"$ is the notion of stochastic domination given in Definition \ref{def:2.3} below. 

By computing the high moments of $N^{-1}\tr G -N^{-1}\bb E \tr G$ using cumulant expansion/Schur complement formula, it can be proven, as previously in \cite{EKYY1} that 
\begin{equation*}
\frac{1}{N}\tr G(z) -\frac{1}{N}\bb E \tr G(z) \prec \frac{1}{N\eta}+\frac{1}{q^2}\,.
\end{equation*}
In order to improve the second term $1/q^2$ to the optimal scale $1/(\sqrt{N}q)$, we need more expansions. However, each additional expansion, in the worst case, only results in an improvement of factor $1/q^2$. When $q=N^{\epsilon}$, it is impossible to write down each expansion explicitly, and one has to introduce general formulas that allows recursive expansions. In order to do so, we implement the ideas in  \cite{HK2}, to construct a hierarchy of Schwinger-Dyson equations for a sufficiently large class of polynomials in the entries of the Green function. As \cite{HK2} deals with the covariance of two Green functions of Wigner matrices, we also need to adapt the method to our current setting, which deals with high-moment estimates of Green functions of sparse matrices. See Section \ref{sec4.2} for more details.

We also apply Proposition \ref{propgreen} to prove the following CLT for mesoscopic linear statistics of Green functions.
\begin{theorem} \label{thmlinstat}
	 Let $z =E+\ii \eta \in {\bf D}_\tau$.
	\begin{enumerate}
		\item When $\eta \gg q/\sqrt{N}$,
		\begin{equation} \label{111}
		\frac{1}{m(z)m'(z)\sqrt{2\bb EH^4_{12}}
			\,N}\big(\tr G(z) -\bb E \tr G(z)\big) \overset{d}{\longrightarrow}\cal N(0,1)\,,
		\end{equation}
		where $m$ is the Stieltjes transform  of the Wigner semicircle law.
		\item When $\eta \ll q/\sqrt{N}$,
		\begin{equation} \label{222}
	  \sqrt{2}\eta\big(\tr G(z) -\bb E \tr G(z)\big) \overset{d}{\longrightarrow}\cal N_{\bb C}(0,1)\,.
		\end{equation}
		Here $\cal N_{\bb C}(0,1)$ denotes the distribution of the standard complex Gaussian random variable.
	\end{enumerate}
\end{theorem}

Note that \eqref{222} coincides with the mesoscopic linear statistics for GOE \cite{Kho1}, whose source is the extrapolation of WDM (or sine-kernel) statistics to mesoscopic scales. On the other hand, \eqref{111} comes from the sparsity of $H$. Thus our result shows that, although the eigenvalue statistics for sparse matrices are different from WDM statistics on large scales, WDM statistics remain valid on small enough mesoscopic scales. This bridges the results on microscopic \cite{EKYY2,HLY15} and macroscopic \cite{ST12,BG13} statistics of $H$.

The rest of the paper is organized as follows. In Section \ref{sec2.5} we introduce the notations and previous results that we use in this paper. In Section \ref{sec4} we prove our main results, Theorems \ref{mainthm} -- \ref{thmlinstat}, assuming a key result on centered moments of mesoscopic linear statistics, Proposition \ref{propgreen}. In
Section \ref{section4} we introduce a class of polynomials in the entries of the Green function, and construct a hierarchy of its Schwinger-Dyson equations. We then use this construction to prove Proposition \ref{propgreen}. Finally in Section \ref{sec5} we prove the general estimates for the class of polynomials of Green function that we used in Section \ref{section4}.

\subsection*{Conventions}
Throughout this paper, we regard $N$ as our fundamental large parameter. Any quantities that are not explicitly constant or fixed may depend on $N$; we almost always omit the argument $N$ from our
notation. We use $\tau$ to denote some generic (small) positive constant, whose value may change from one expression to the next. Similarly, we use $C$ to denote some generic (large) positive constant. For $A, B>0$, we use $A=O(B)$ to denote $A \leq  CB$ and $A\asymp B$ to denote $C^{-1}B\leq A \leq CB$. When we write $A \ll B$ and $A\gg B $, we mean $A \leq CN^{-\tau}B$ and $A \geq C^{-1} N^{\tau}B$ for some constants $C,\tau>0$ respectively.

\subsection*{Acknowledgment}
The author would like to thank Gaultier Lambert for helpful discussions, and thank Antti Knowles and Benjamin Schlein for many useful comments on the preliminary draft. This project has received funding from NCCR Swissmap, the Swiss National Science Foundation (SNF) grant No.\ 20020\_172623, and the European Research Council (ERC) under the European Union’s Horizon 2020 research and innovation programme (grant agreement No.\ 715539\_RandMat).

\nc	
\section{Preliminaries} \label{sec2.5}

In this section we collect notations and tools that are used in the paper.

Let $M$ be an $N \times N$ matrix. We denote $M^{*n}\deq (M^{*})^n$, $M^{*}_{ij}\deq (M^{*})_{ij} = \ol M_{ji}$, $M^n_{ij}\deq (M_{ij})^n$, and the normalized trace of $M$ by $\ul M \deq \frac{1}{N} \tr M$. We abbreviate $\langle X \rangle  \deq X-\bE X$ for any random variable $X$ with finite expectation. For the Green function $G$, we have the differential rule
\begin{equation} \label{diff}
\frac{\partial G_{ij}}{\partial H_{kl}}=-(G_{ik}G_{lj}+G_{il}G_{kj})(1+\delta_{kl})^{-1}\,.
\end{equation}Let $\mu$ be the empirical spectral measure of $H$. Its Stieltjes transform is denoted by
\begin{equation*}
\underline{G}(z)\deq \frac{1}{N}\tr G(z)=\int \frac{\mu(x)}{x-z}\,\dd x\,.
\end{equation*}
We also have
\begin{equation} \label{diffH}
\frac{\partial (\ul{H^2}-1)}{\partial H_{ij}} =\frac{4}{N}H_{ij}(1+\delta_{ij})^{-1} \prec \frac{1}{Nq} \quad \mbox{and} \quad \frac{\partial^2 (\ul{H^2}-1)}{\partial H^2_{ij}} =\frac{4}{N}(1+\delta_{ij})^{-1} \,.
\end{equation}
For $z \in \bC$ with $\im z \ne 0$, the Stieltjes transform  of the Wigner semicircle law is defined by 
\begin{equation*} 
 m(z)\deq \int \frac{\varrho(x) }{x-z}\,\dd x\,.
\end{equation*}
One elementary fact is that $m$ is the unique solution of
\begin{equation} \label{mmm}
1+zm(z)+m(z)^2=0
\end{equation}
satisfying $\im m(z) \im z >0$. Let us define the spectral domains 
$$ 
{\bf S} = \{E+\mathrm{i}\eta: |E| \le 4, 0 < \eta \le 4 \} \quad \mbox{and} \quad \widetilde{\bf S}_{\tau}=\{E+\mathrm{i}\eta: |E| \le 4, N^{-1+\tau} \leq \eta \le 4\}\,.
$$
We denote the distance to spectral edge by 
\[
\kappa\equiv\kappa_E\deq\min\{|2-E|,|2+E|\}\,.
\]
\begin{lemma}[Basic properties of $m$] \label{lem:m} 
	We have
	\[
	|m(z)|\asymp 1 \quad\quad \mbox{and} \quad\quad
	|m'(z)|\asymp \frac{1}{\sqrt{\kappa+\eta}}
	\]
	for all $z \in \b S$. In particular, $|m'(z)|\asymp 1$ for all $z \in \b D_{\tau}$.
\end{lemma}
\begin{proof}
	The proof is an elementary exercise using \eqref{mmm}.
\end{proof}

If $h$ is a real-valued random variable with finite moments
of all order, we denote by $\cal C_k(h)$ the $k$th cumulant of $h$, i.e.
\[
\cal C_k(h)\deq (-\ii)^k \cdot\big(\partial_{\lambda}^k \log \bb E \e^{\ii  \lambda h}\big) \big{|}_{\lambda=0}\,.
\]

We state the cumulant expansion formula, whose proof is given in e.g. \cite[Appendix A]{HKR}. 
\begin{lemma}[Cumulant expansion] \label{lem:cumulant_expansion}
	Let $f:\R\to\C$ be a smooth function, and denote by $f^{(k)}$ its $k$th derivative. Then, for every fixed $\ell \in\N$, we have 
	\begin{equation}\label{eq:cumulant_expansion}
	\mathbb{E}\big[h\cdot f(h)\big]=\sum_{k=0}^{\ell}\frac{1}{k!}\mathcal{C}_{k+1}(h)\mathbb{E}[f^{(k)}(h)]+\cal R_{\ell+1},
	\end{equation}	
	assuming that all expectations in \eqref{eq:cumulant_expansion} exist, where $\cal R_{\ell+1}$ is a remainder term (depending on $f$ and $h$), such that for any $t>0$,
	\begin{equation*} 
	\cal R_{\ell+1} = O(1) \cdot \bigg(\E\sup_{|x| \le |h|} \big|f^{(\ell+1)}(x)\big|^2 \cdot \E \,\big| h^{2\ell+4} \mathbf{1}_{|h|>t} \big| \bigg)^{1/2} +O(1) \cdot \bb E |h|^{\ell+2} \cdot  \sup_{|x| \le t}\big|f^{(\ell+1)}(x)\big|\,.
	\end{equation*}
\end{lemma}

The following result gives bounds on the cumulants of the entries of $H$, whose proof follows by the homogeneity of the cumulants.
\begin{lemma} \label{Tlemh}
	For every $k \in \bb N$ we have
	\begin{equation*}
	\cal C_{k}(H_{ij})=O_{k}(1/(Nq^{k-2}))
	\end{equation*}
	uniformly for all $i,j$.
\end{lemma}

The following is a standard complex analysis result from \cite{Davies}. 

\begin{lemma}[Helffer-Sj\"{o}strand formula] \label{HS}
	Let $f \in {C}^{2}(\bR)$, and let $\tilde{f}$ be the almost analytic extension of $f$ defined by
	\begin{equation*}
	\tilde{f}(x+\ii y)\deq f(x) +\ii yf^{\prime}(x)\,.
	\end{equation*}
	Let $\chi \in \cal C^{\infty}_c(\R)$ be a cutoff function satisfying $\chi(0) = 1$, and by a slight abuse of notation write $\chi(z) \equiv \chi (\im z)$.
	Then for any $\lambda \in \bR$ we have
	\begin{equation*} 
	f(\lambda)=\frac{1}{\pi}\int_{\bC}\frac{\partial_{\bar{z}}(\tilde{f}(z)\chi(z))}{\lambda-z}\,\dd^2z\,,
	\end{equation*}
	where $\partial_{\bar{z}}\deq \frac{1}{2}(\partial_x+\mathrm{i}\partial_y)$ is the antiholomorphic derivative and $\dd^2 z$ the Lebesgue measure on $\C$.
\end{lemma}

The following definition introduces a (conventional) notion of a high-probability bound that is used commonly in random matrix theory. 
\begin{definition}[Stochastic domination] \label{def:2.3} 
	Let $$X=\pb{X^{(N)}(u): N \in \bN, u \in U^{(N)}}\,,\qquad Y=\pb{Y^{(N)}(u): N \in \bN, u \in U^{(N)}}$$ be two families of random variables, where $Y^{(N)}(u)$ are nonnegative and $U^{(N)}$ is a possibly $N$-dependent parameter set. We say that $X$ is stochastically dominated by $Y$, uniformly in $u$, if for all (small) $\varepsilon>0$ and (large) $D>0$ we have
	\begin{equation*} 
	\sup\limits_{u \in U^{(N)}}	\bP \left[ \big|X^{(N)}(u)\big| > N^{\varepsilon} Y^{(N)}(u) \right] \le N^{-D}
	\end{equation*} 
	for large enough $N \ge N_0(\varepsilon,D)$. If $X$ is stochastically dominated by $Y$, uniformly in $u$, we use the notation $X \prec Y$, or equivalently $X=O_{\prec}(Y)$. (Note that for
	deterministic $X$ and $Y$, $X =O_\prec(Y)$ means $X= O_{\epsilon}(N^{\epsilon}Y)$ for any $\epsilon> 0$.)
\end{definition}

 Next we recall the local semicircle law for Erd\H{o}s -R\'enyi graphs in \cite{EKYY1}.

\begin{theorem}[Theorem 2.8,\cite{EKYY1}] \label{refthm1}
	Let $H$ be a sparse matrix defined as in Definition \ref{def:sperse}. We have
	\begin{equation*} 
	\max\limits_{i,j}|G_{ij}(z)-\delta_{ij}m(z)| \prec \frac{1}{q}+\sqrt{\frac{\im m(z)}{N\eta}}+
\frac{1}{N\eta}
	\end{equation*} 
	and 
	\begin{equation*}
	|\ul{G}-m |\prec \frac{1}{q}\wedge\frac{1}{q^2(\eta+\kappa_E)}+\frac{1}{N\eta}
	\end{equation*}
	uniformly in $z =   
	E+\mathrm{i}\eta \in \bf S$. 
\end{theorem}

\begin{remark}
	Theorem \ref{refthm1} was proved in \cite{EKYY1} under the additional assumption $\bb E H^2_{ii}=1/N$ for all $i$. However, the proof is insensitive to the variance of the diagonal entries, and one can easily repeat the steps in \cite{EKYY1} under the general assumption $\bb E H_{ij}^2=C_i/N$. A weak local law for $H$ with general variances on the diagonal can also be found in \cite{HKM18}.
\end{remark}

We also need the following result from \cite{EKYY1} concerning the density of states of $A$.

\begin{lemma}[Theorem 2.10,\cite{EKYY1}]  \label{cor2} 
Let $\widetilde{\mu}$ be the empirical eigenvalue density of $A$. For any interval $I\subset \bb R$, we have
	\[
	\big|\widetilde{\mu}(I)-\varrho(I) \big| \prec \frac{1}{N}+\frac{|I|}{q}\,.
	\]
	\end{lemma}
We recall the magical Ward identity.
\begin{lemma}[Ward identity] \label{Ward}
	We have 
	\begin{equation*}
	\sum_j |G_{ij}|^2 =\frac{\im G_{ii}}{\eta}
	\end{equation*}
	for all $z=E+\ii \eta \in \b S$.
\end{lemma}

	Finally, we collect some estimates in the following lemma, whose proof is postponed to Appendix \ref{appA}.
	\begin{lemma}\label{prop4.4}
	(i)	For any fixed $m,n \in \bb N$ such that $m+n\geq 1$, we have
		\begin{equation}   \label{Tk}
		\big\langle \ul{G^{m}G^{*n}} \big\rangle \prec \eta^{1-(m+n)}\bigg(\frac{1}{q}+\frac{1}{N\eta}\bigg)
		\end{equation}
		as well as
		\begin{equation} \label{410}
		\big(G^mG^{*n}\big)_{ij}\prec
		\begin{cases}
		\eta^{1-(m+n)} & \txt{if } i = j
		\\
		\eta^{1-(m+n)}\Big(\frac{1}{q}+\frac{1}{\sqrt{N\eta}}\Big) & \txt{if } i \neq j\,,
		\end{cases}
		\end{equation}
		uniformly in $i,j$ and $z=E+\ii \eta \in \widetilde{\bf S}_{\tau}$.
		
		(ii) For any fixed $m,n \in \bb N$ such that $m+n\geq 1$, we have
		\begin{equation*}
		\sum_{i} | \big(G^mG^{*n}\big)_{ij}|^2 \prec \eta^{1-2(m+n)}
		\end{equation*}
		uniformly in $j$ and $z=E+\ii \eta \in \widetilde{\bf S}_{\tau}$.
		
		(iii) For $k=2,3$, we have
		\begin{equation*}
		\bb E\ul{G^k} \prec \Big(\frac{1}{q}+\frac{1}{N\eta}+\eta\Big) \eta^{1-k}
		\end{equation*}
		uniformly for all $z=E+\ii \eta \in {\bf D}_{\tau}$.
		
		(iv) For any fixed $n \in \bb N_+$,
		\begin{equation} \label{H is gaussian}
		\bb E (\ul{H^2}-1)^n =\begin{cases}
		(n-1)!!\, (2\bb E H_{12}^4)^{n/2}+O\left(\frac{1}{(\sqrt{N}q)^n}\cdot \frac{q}{\sqrt{N}}\right) \quad &\mbox{ if $n$ is even } \vspace{0.2cm}\\
		O\left(\frac{1}{(\sqrt{N}q)^n}\cdot \frac{q}{\sqrt{N}}\right) \quad &\mbox{ if $n$ is odd. }
		\end{cases}
		\end{equation}
	\end{lemma}

\section{Proof of main results} \label{sec4}

For $z =E+\ii \eta\in {\bf D}_{\tau}$, we write
\begin{equation*}
\alpha\deq -\log_N \eta
\end{equation*}
so that $\eta=N^{-\alpha}, \alpha \in [0,1-\tau]$. We define
\begin{equation} \label{gamma'}
\delta\equiv \delta(z)\deq \min\Big\{\beta,\frac{1}{2}-\beta,\frac{1-\alpha}{2}\Big\}>0\,.
\end{equation}
and
\begin{equation} \label{chi}
\xi\equiv \xi(z)=\frac{1}{8}\min\Big\{\frac{\alpha}{2},\delta\Big\}\geq 0\,,
\end{equation} 
We define the linear statistics with a random shift
\begin{equation*}
[G]\equiv [G(z)]\deq \frac{1}{N}\tr G(z) -\frac{1}{N}\bb E \tr G(z) - (\ul{H^2}-1) m(z)m'(z)\,.
\end{equation*}
Note that \eqref{H is gaussian} implies
\begin{equation} \label{bound for H^2}
\frac{1}{\sqrt{2\bb E H_{12}^4}}(\ul{H^2}-1) \overset{d}{\longrightarrow}\cal N(0,1) \quad \mbox{and} \quad \ul{H^2}-1 \prec N^{-1/2-\beta}\,.
\end{equation}
The term $\ul{H^2}-1$ was introduced in \cite{HLY} to study the eigenvalue fluctuations of $A$ near the edge.  

In this section we shall prove Theorems \ref{mainthm}--\ref{thmlinstat} assuming the following proposition, whose proof is postponed to Section \ref{section4}.

\begin{proposition} \label{propgreen}
	Let $m,n \in\bb N_+$. We have
	\begin{equation} \label{result}
	\bb E |[G]|^{2n}=\frac{n!}{2^n}\Big(\frac{1}{N\eta}\Big)^{2n}+O_{\prec}\bigg(\frac{N^{-\xi}}{(N\eta)^{2n}}+\bigg(\frac{N^{-\delta/4}}{\sqrt{N}q}\bigg)^{2n}\bigg)
	\end{equation}
	and
	\begin{equation} \label{resultt}
	\bE [G^*]^n[G]^m=O_{\prec}\bigg(\frac{N^{-\xi}}{(N\eta)^{m+n}}+\bigg(\frac{N^{-\delta/4}}{\sqrt{N}q}\bigg)^{m+n}\bigg)
	\end{equation}
	for $m \ne n$, uniformly for all $z \in {\bf D}_{\tau}$. 
\end{proposition}
We observe that Theorem \ref{thmlinstat} is an immediate consequence of Lemma \ref{lem:m}, Lemma \ref{prop4.4} (iv) and Proposition \ref{propgreen}. One can follow, e.g.\,the steps in \cite{HK}, to show Theorem \ref{thmlinstat} for general test functions. We do not pursue it here.
\subsection{Proof of Theorem \ref{cor1.4}}
In this section, we prove the following result, which trivially implies Theorem \ref{cor1.4} by \eqref{bound for H^2}.
\begin{proposition} \label{prop3.2}
	Let $E \in [-2+\tau,2-\tau]$. We have
	\[
	\Sigma(E)-\bb E \Sigma(E)-\frac{E\sqrt{4-E^2}}{4\pi} (\ul{H^2-1})N \prec N^{1/2-\beta-\zeta/16}\,.
	\]
\end{proposition}
\begin{proof}
Let $f \in C^{\infty}(\bb R)$ such that $f=1$ in $(-3+\frac{1}{N},E-\frac{1}{N}]$ and $f=0$ in $(-\infty,-3-\frac{1}{N}]\cup[E+\frac{1}{N},+\infty)$. We further assume $|f'|=O(N)$ and $|f''|=O(N^2)$. Let us write $z=x+\ii y$ and choose $\chi\equiv \chi(y)$ such that $\chi(y)=1$ for $|y|\leq 1$ and $\chi(y)=0$ for $|y|\geq 2$. Note that by Green's theorem we have
\[
\int_{\bb C} \partial_{\bar{z}} (\tilde{f}(z)\chi(z)) m(z)m'(z)\, \dd^2 z= \frac{1}{2}\int_{-2}^2 f(x) \frac{2-x^2}{\sqrt{4-x^2}} =\frac{E\sqrt{4-E^2}}{4}+O(N^{-1})\,,
\]
and Lemma \ref{HS} implies
\[
\tr f(H) -\bb E \tr f(H)=\frac{N}{\pi} \int_{\bb C} \partial_{\bar{z}} (\tilde{f}(z)\chi(z)) \langle \ul{G}\rangle \, \dd^2 z\,.
\] 
Combining the above two relations, and together with \eqref{bound for H^2}, we have
\begin{equation} \label{5.1}
\tr f(H) -\bb E \tr f(H)-\frac{E\sqrt{4-E^2}}{4\pi} (\ul{H^2-1})N  = \frac{N}{\pi} \int_{\bb C} \partial_{\bar{z}} (\tilde{f}(z)\chi(z)) [G(z)]\, \dd^2 z+O_{\prec}(N^{-1/2-\beta})\,.
\end{equation}
Recall the definition of $\zeta>0$ from \eqref{zeta}. By \eqref{result} we have
\begin{equation*}
[G(z)] \prec \frac{1}{Ny}+\frac{N^{-\zeta/16}}{\sqrt{N}q}
\end{equation*}
uniformly for $z=x+\ii y \in {\bf D}_{\zeta/2}$, and an $N^{-3}$-net argument \cite[Remark 2.7]{BK16} shows
\begin{equation}  \label{38}
\sup_{z \in \b D_{\zeta/2}} |[G(z)] |  \Big(\frac{1}{Ny}+\frac{N^{-\zeta/16}}{\sqrt{N}q}\Big)^{-1} \prec 1\,.
\end{equation}
Fix $\varepsilon>0$. By Theorem \ref{refthm1}, \eqref{bound for H^2} and an $N^{-3}$-net argument, we see that
\begin{equation} \label{399}
\sup_{z \in \b S,\, N^{-1+\varepsilon}\leq y \leq N^{-1+\zeta/2}} |Ny[G(z)]| \prec \sum_{N^{-1+\varepsilon}\leq y \leq N^{-1+\zeta/2}}Ny\Big(\frac{1}{q}+\frac{1}{Ny}+\frac{1}{\sqrt{N}q}\Big) \prec 1\,.
\end{equation}
From Theorem \ref{refthm1}, $|m(z)| \leq C$, and an $N^{-3}$-net argument, we have
\[
\sup_{|x|\leq 4}\max_{i,j}|G_{ij}(x+\ii N^{-1+\varepsilon})| \prec 1\,,
\]
and a deterministic monotonicity result \cite[Lemma 10.2]{BK16} shows
\[
\sup_{z =x+\ii y\in \b S,\, y \leq N^{-1+\varepsilon}}\max_{i,j}|NyG_{ij}(z)| \prec N^{\varepsilon}\,.
\]
Using the above relation, together with \eqref{bound for H^2}, \eqref{399} and the arbitrarity of $\varepsilon$, we have
\begin{equation} \label{39}
\sup_{z \in \b S,\, y \leq N^{-1+\zeta/2}} |Ny[G(z)]| \prec 1\,.
\end{equation}
We split
\begin{multline} \label{310}
\bigg| \frac{N}{\pi} \int_{\bb C} \partial_{\bar{z}} (\tilde{f}(z)\chi(z)) [G(z)]\, \dd^2 z\bigg|\leq C \bigg| N \int_{1\leq y \leq 2} (f(x)+\ii y f'(x))\chi'(y) [G(x+\ii y)]\, \dd^2 z\bigg|\\
+ C \bigg| N \int_{0<y \leq N^{-1+\zeta/2}} f''(x)y\chi(y) [G(z)]\, \dd^2 z\bigg|+C \bigg| N \int_{ N^{-1+\zeta/2}\leq y \leq 2} f''(x)y\chi(y) [G(z)]\, \dd^2 z\bigg|\,.
\end{multline} 

By \eqref{38} and $\|f\|_1+\|f'\|_1\leq C$ we have
\begin{equation} \label{dingman}
 \bigg| N \int_{1\leq y \leq 2} (f(x)+\ii y f'(x))\chi'(y) [G(x+\ii y)]\, \dd^2 z\bigg| \prec N\bigg(\frac{1}{N}+\frac{N^{-\zeta/16}}{\sqrt{N}q}\bigg) \prec N^{1/2-\beta-\zeta/16}\,.
\end{equation}
By \eqref{39} and $\|f''\|_1=O(N)$ we have
\begin{equation}
\bigg| N \int_{0<y \leq N^{-1+\zeta/2}} f''(x)y\chi(y) [G(z)]\, \dd^2 z\bigg| \prec  N \cdot \int_{0}^{N^{-1+\zeta/2}} 1\, \dd y= N^{\zeta/2} \leq N^{1/2-\beta-\zeta/2}\,.
\end{equation}
For the last term on RHS of \eqref{310}, we do integration by parts, first in $x$ and then in $y$, and get
\begin{multline*}
\bigg| N \int_{ N^{-1+\zeta/2}\leq y \leq 2} f''(x)y\chi(y) [G(z)]\, \dd^2 z\bigg|\leq \bigg| N \int_{ N^{-1+\zeta/2}\leq y \leq 2} f'(x)\chi(y) [G(z)]\, \dd^2 z\bigg|\\
+\bigg| N \int_{ 1\leq y \leq 2} f'(x)y\chi'(y) [G(z)]\, \dd^2 z\bigg|+\bigg| N \int f'(x)N^{-1+\zeta/2}\chi(N^{-1+\zeta/2}) [G(x+\ii N^{-1+\zeta/2})]\, \dd x\bigg|\,,
\end{multline*}
and again by \eqref{38} we have
\begin{equation} \label{312}
\bigg| N \int_{ N^{-1+\zeta/2}\leq y \leq 2} f''(x)y\chi(y) [G(z)]\, \dd^2 z\bigg| \prec N^{1/2-\beta-\zeta/16}\,.
\end{equation}
From \eqref{5.1}, \eqref{310}--\eqref{312} we have
\begin{equation} \label{3.14159}
\tr f(H) -\bb E \tr f(H)-\frac{E\sqrt{4-E^2}}{4\pi} (\ul{H^2-1})N  =O_{\prec}(N^{1/2-\beta-\zeta/16})\,.
\end{equation}
Let $\widetilde{\Sigma}$ be the eigenvalue counting function of $H$. Note that $\|H\|\leq 5/2$ with overwhelming probability.  Thus
	\begin{multline*}
	\Big|\tr f(H)- \widetilde{\Sigma}(E) \Big|\leq C \Big(\widetilde{\Sigma}(E+N^{-1})-\widetilde{\Sigma}(E-N^{-1})+\widetilde{\Sigma}(-3+N^{-1})\Big)\\
	=C \Big(\widetilde{\Sigma}(E+N^{-1})-\widetilde{\Sigma}(E-N^{-1})\Big)+O_{\prec}(1)\,.
	\end{multline*}
From Theorem \ref{refthm1} we know
	\[
   \widetilde{\Sigma}(E+N^{-1})-\widetilde{\Sigma}(E-N^{-1})\leq\sum_{i} \frac{2}{N^2(\lambda_i-E)^2+1}=2\im \ul{G}(E+\ii N^{-1}) \prec 1\,,
	\]
	and thus
	\begin{equation} \label{3.1}
	\tr f(H)- \widetilde{\Sigma}(E) \prec 1\,.
	\end{equation}
	Note that \eqref{3.1} also implies
	\begin{equation} \label{3.2}
	\bb E\tr f(H)- \bb E\widetilde{\Sigma}(E)  \prec 1\,.
	\end{equation}
	By Cauchy interlacing theorem (e.g. \cite[Lemma 6.1]{EKYY2}) we have 
	\begin{equation} \label{eqn3.6}
	\widetilde{\Sigma}(E)-1 \leq \Sigma(E)\leq \widetilde{\Sigma}(E)\,.
	\end{equation} 
	By \eqref{3.1} -- \eqref{eqn3.6} we have
	\begin{equation} \label{3.33}
\tr f(H)-\bb E \tr f(H)-({\Sigma}(E)-\bb E {\Sigma}(E)) \prec 1\,.
	\end{equation}
Combining \eqref{3.14159} and \eqref{3.33} finished the proof.
\end{proof}

\subsection{Proof of Theorem \ref{mainthm}}
We shall prove the following result, and Theorem \ref{mainthm} then follows by Lemma \ref{prop4.4} (iv).
\begin{proposition}
	Fix $\tau>0$. We denote the eigenvalues of $A$ by $\lambda_1 \leq \lambda_2 \leq \dots \leq \lambda_N$. For all $i \in [\tau N,(1-\tau)N]$, we have
	\begin{equation*} 
	\lambda_i-\bb E \lambda_i -\frac{\gamma_i}{2} (\ul{H^2}-1)  \prec  N^{-1/2-\beta-\zeta/16}\,.
	\end{equation*}
\end{proposition}

\begin{proof}
Let $\widetilde{\mu}$ be the empirical eigenvalue density of $A$. Let us define the function $g: \bb R \to \bb R$ by
\[
g(a)\deq \bb E \int_{-\infty}^{a} \widetilde{\mu}(x) \,\dd x\,.
\] 
We claim that for any fixed (small) $\epsilon>0$, $g$ has no jumps of size larger than $N^{-1+\epsilon}$. In fact, by Lemma \ref{cor2} we have
\[
g(a+N^{-1+\epsilon/2})-g(a)=\bb E\widetilde{\mu}((a,a+N^{-1+\epsilon/2}])=\varrho((a,a+N^{-1+\epsilon/2}])+O_{\prec}(N^{-1})\leq CN^{-1+\epsilon/2}\,.
\]
Pick $i \in [\tau N,(1-\tau)N]$. We can then choose deterministic $\theta_i \in \bb R$ satisfying
\[
\bigg|\bb E \int_{-\infty}^{\theta_i} \widetilde{\mu}(x) \,\dd x-\frac{i}{N}\bigg|\leq N^{-1/2-\beta-\zeta/2}\,, 
\] 
so that $|\bb E \Sigma(\theta_i)-i|\leq N^{1/2-\beta-\zeta/2}$. Fix $\varepsilon\in (0,\zeta/16)$. Let us abbreviate $\omega_i=\theta_i-N^{-1/2-\beta-\zeta/16+\epsilon}$. We have
\begin{multline} \label{3.4}
\bb P \Big(\lambda_i-\theta_i-(\ul{H^2}-1)\gamma_i/2\leq -N^{-1/2-\beta-\zeta/16+\epsilon}\Big)=\bb P\Big(\Sigma\big(\omega_i+(\ul{H^2}-1)\gamma_i/2\big) \geq i \Big)\\
	=\bb P\Big(\Sigma\big(\omega_i+(\ul{H^2}-1)\gamma_i/2\big)-\bb E\Sigma\big(\omega_i\big)  \geq \bb E \Sigma(\theta_i)-\bb E\Sigma(\omega_i)+O(N^{1/2-\beta-\zeta/2})\Big)
\end{multline}
By Lemma \ref{cor2} we know that
\begin{equation} \label{3.200}
\widetilde{\mu}(I)-\varrho(I) \prec N^{-1/2-\beta-\zeta}
\end{equation}
for any $I$ satisfying $|I| \prec N^{-1/2-\beta}$. Together with \eqref{bound for H^2} we have
\begin{multline} \label{3.20}
\Sigma\big(\omega_i+(\ul{H^2}-1)\gamma_i/2\big)-\Sigma(\omega_i)=N\varrho([\omega_i,\omega_i+(\ul{H^2}-1)\gamma_i/2])+O_{\prec}(N^{1/2-\beta-\zeta})\\
=\frac{\sqrt{4-\omega_i^2}}{4\pi}N(\ul{H^2}-1)\gamma_i+O_{\prec}(N^{1/2-\beta-\zeta})
=\frac{\sqrt{4-\omega_i^2}}{4\pi}N(\ul{H^2}-1)\omega_i+O_{\prec}(N^{1/2-\beta-\zeta})\,.
\end{multline}
In the last step of \eqref{3.20} we used $|\theta_i -\gamma_i| \prec N^{-\zeta}$, which also can be deduced from Lemma \ref{cor2}. By \eqref{3.200},
\begin{equation} \label{not imop}
\bb E \Sigma(\theta_i)-\bb E\Sigma(\omega_i)=N\varrho([\omega_i,\theta_i])+O_{\prec}(N^{1/2-\beta-\zeta})=\frac{\sqrt{4-\omega_i^2}}{2\pi}N^{1/2-\beta-\zeta/16+\epsilon}+O_{\prec}(N^{1/2-\beta-\zeta})\,.
\end{equation}
A combination of \eqref{3.4}, \eqref{3.20}, and \eqref{not imop} shows that
\begin{multline*} 
\bb P \Big(\lambda_i-\theta_i-(\ul{H^2}-1)\gamma_i/2\leq -N^{-1/2-\beta-\zeta/16+\epsilon}\Big)\\
= \bb P\bigg(\Sigma\big(\omega_i\big)-\bb E\Sigma\big(\omega_i\big) -\frac{\sqrt{4-\omega_i^2}}{4\pi}N(\ul{H^2}-1)\omega_i \geq \frac{\sqrt{4-\omega_i^2}}{2\pi}N^{1/2-\beta-\zeta/16+\epsilon}+O_{\prec}(N^{1/2-\beta-\zeta/2})\bigg)\,.
\end{multline*}
Since $\epsilon$ is arbitrary, by Proposition \ref{prop4.4} we see that
\begin{equation*}
\big(\lambda_i-\theta_i-(\ul{H^2}-1)\gamma_i/2 \big)_- \prec N^{-1/2-\beta-\zeta/16}\,.
\end{equation*}
Repeating the above process for
\[
\bb P \Big(\lambda_i-\theta_i-(\ul{H^2}-1)\gamma_i/2\geq N^{-1/2-\beta-\zeta/16+\epsilon}\Big)
\]
we can also show that
\begin{equation*}
\big(\lambda_i-\theta_i-(\ul{H^2}-1)\gamma_i /2\big)_+ \prec N^{-1/2-\beta-\zeta/16}\,.
\end{equation*}
Thus 
\begin{equation} \label{123}
\lambda_i-\theta_i-(\ul{H^2}-1)\gamma_i/2  \prec N^{-1/2-\beta-\zeta/16}\,,
\end{equation}
which also implies 
\begin{equation}\label{456}
\bb E\lambda_i-\theta_i \prec N^{-1/2-\beta-\zeta/16}\,.
\end{equation}
The proof then follows from \eqref{123} and \eqref{456}.
\end{proof}

\section{Proof of Proposition \ref{propgreen}} \label{section4}
In this section we prove \eqref{result}; the proof of \eqref{resultt} is similar, and we omit the details. Throughout this section let us pick $n \in \bb N_+$ and
\begin{equation}\label{z}
  z=E+\ii \eta \in {\bf D}_{\tau} \,.
  \end{equation}
Let us define
$$
\cal M\deq \|[G]\|_{2n}=\Big(\bE \big|[G]\big|^{2n}\Big)^{\frac{1}{2n}}\,,
$$ 
and	we split
	\begin{equation*}
	\cal M^{2n} =\bE [G^*]^n[G]^{n-1} \langle \ul{G}\rangle-\bE [G^*]^n[G]^{n-1} (\ul{H^2}-1) mm'\,.
	\end{equation*}
The proof of \eqref{result} is immediate from the next lemma.
  \begin{lemma} \label{lemmakey}
	We have 
    \begin{equation} \label{est1}
		\bE [G^*]^n[G]^{n-1} \langle \ul{G}\rangle =\frac{n}{2N^2\eta^2}\bb E|[G]|^{2n-2}+O_{\prec} \big(N^{-\delta}\big)\cal M^{2n}+\sum_{r=1}^{2n}O_{\prec}\bigg(\frac{N^{-\xi}}{(N\eta)^{r}}+\bigg(\frac{N^{-\delta/4}}{\sqrt{N}q}\bigg)^{r}\bigg)\cal M^{2n-r}\,,
		\end{equation}
and
   \begin{equation} \label{est2}
		\bE [G^*]^n[G]^{n-1} (\ul{H^2}-1) mm'=\sum_{r=1}^{2n}O_{\prec}\bigg(\frac{N^{-\xi}}{(N\eta)^{r}}+\bigg(\frac{N^{-\delta/4}}{\sqrt{N}q}\bigg)^{r}\bigg)\cal M^{2n-r}\,.
	\end{equation}
	\end{lemma}
In fact, Lemma \ref{lemmakey} shows
\begin{equation} \label{uzi}
\cal M^{2n}=\frac{n}{2N^2\eta^2}\bb E|[G]|^{2n-2}+O_{\prec} \big(N^{-\delta}\big)\cal M^{2n}+\sum_{r=1}^{2n}O_{\prec}\bigg(\frac{N^{-\xi}}{(N\eta)^{r}}+\bigg(\frac{N^{-\delta/4}}{\sqrt{N}q}\bigg)^{r}\bigg)\cal M^{2n-r}\,,
\end{equation}
and together with $\bb E |[G]|^{2n-2}\leq \cal M^{2n-2}$ we have
\[
\cal M^{2n} \prec \sum_{r=1}^{2n}\bigg(\frac{1}{(N\eta)^{r}}+\bigg(\frac{N^{-\delta/4}}{\sqrt{N}q}\bigg)^{r}\bigg)\cal M^{2n-r}\,,
\]
which implies
\[
\cal M^{2n} \prec \frac{1}{(N\eta)^{2n}}+\bigg(\frac{N^{-\delta/4}}{\sqrt{N}q}\bigg)^{2n}\,.
\]
Since $n$ is arbitrary, we have
\begin{equation} \label{baqiang}
[G] \prec \frac{1}{N\eta}+\frac{N^{-\delta/4}}{\sqrt{N}q}\,.
\end{equation}
Inserting \eqref{baqiang} back into \eqref{uzi}, we have
\[
\bb E |[G]|^{2n}=\cal M^{2n}=\frac{n}{2N^2\eta^2}\bb E|[G]|^{2n-2}+O_{\prec}\bigg(\frac{N^{-\xi}}{(N\eta)^{2n}}+\bigg(\frac{N^{-\delta/4}}{\sqrt{N}q}\bigg)^{2n}\bigg)\,,
\]
and \eqref{result} follows by iteration.

In Sections \ref{section4.1}--\ref{sec4.5} we shall prove \eqref{est1}, and in Section \ref{sec4.6} we prove \eqref{est2}.

	\subsection{First estimates} \label{section4.1}
	 By the resolvent identity $zG=GH-I$ we have
	\begin{equation} \label{3.11}
	z\bE [G^*]^n[G]^{n-1} \langle \ul{G}\rangle = \bE \langle [G^*]^n [G]^{n-1} \rangle \underline{GH} = \frac{1}{N}\sum\limits_{i,j} \bE \langle[G^*]^n [G]^{n-1} \rangle G_{ij}H_{ji}\,.
	\end{equation}
	We calculate the RHS of \eqref{3.11} using the cumulant formula \eqref{eq:cumulant_expansion} with $f=f_{ij}(H)\deq \langle[G^*]^n [G]^{n-1} \rangle G_{ij}$ and $h=H_{ji}$, and get
	\begin{equation} \label{4.6}
	\begin{aligned} 
	z\bE [G^*]^n[G]^{n-1} \langle \ul{G}\rangle
	&=\frac{1}{N^2} \sum\limits_{i,j} \bE \langle [G^*]^n [G]^{n-1} \rangle \frac{\partial  G_{ij}}{\partial H_{ji}} (1+\delta_{ji})\\ 
	&\ \ +\frac{1}{N^2} \sum\limits_{i,j} \bE\frac{\partial (\langle [G^*]^n [G]^{n-1}\rangle)}{\partial H_{ji}}G_{ij}(1+\delta_{ji})+
	\bb EK + \sum_{k=2}^{l}\bb EL_k +\frac{1}{N}\sum_{i,j}\bb E\cal R^{(ji)}_{l+1}\\
	&\ \eqd  (a)+(b)+K+ \sum_{k=2}^l L_k+\frac{1}{N}\sum_{i,j}\bb E\cal R^{(ji)}_{l+1}\,,
	\end{aligned} 
	\end{equation}
	where
	\begin{equation} \label{3.13}
	 K=N^{-2} \sum\limits_{i}\frac{\partial (\langle[G^*]^n[G]^{n-1}  \rangle G_{ii})}{\partial H_{ii}}(N\cal C_2(H_{ii})-2)\,,
	\end{equation}
	and
	\begin{equation} \label{3.14} 
	L_k=N^{-1} \cdot \sum\limits_{i,j}\left(\frac{1}{k!}\mathcal{C}_{k+1}(H_{ji})\frac{\partial^k( {\langle [G^*]^n[G]^{n-1}  G_{ij}})}{\partial H_{ji}^k}\right).  
	\end{equation}
	Here $l$ is a fixed positive integer to be chosen later, and $\cal R_{l+1}^{(ji)}$ is a remainder term defined analogously to $\cal R_{l+1}$ in (\ref{eq:cumulant_expansion}). 
	 Using the differential rule \eqref{diff} we get
	\begin{equation*} 
	\begin{aligned}
	(a)&= N^{-2} \sum_{i,j} \bE \langle [G^*]^n[G]^{n-1}  \rangle (-G_{ij}G_{ij}-G_{ii}G_{jj}) \\
	&= -N^{-1}\bE [G^*]^n[G]^{n-1}  \langle\underline{G^2} \rangle - \bE [G^*]^n[G]^{n-1}  \langle \underline{G} \rangle^2 -2\bE [G^*]^n[G]^{n-1} \langle \underline{G} \rangle \bE \underline{G} + \bE [G^*]^n[G]^{n-1} \bE \langle \underline{G} \rangle^2\,.
	\end{aligned} 
	\end{equation*} 
	Similarly, 
	\begin{equation*} 
	(b)=-\frac{2}{N^2}\left(n\bE[G^*]^{n-1}[G]^{n-1} \big(\underline{GG^{*2}}+2\ul{HG^*}\bar{m}\bar{m}'\big)+ (n-1)\bE[G^*]^n[G]^{n-2} \big(\underline{G^3}+2\ul{HG}mm' \big)\right)\,.
	\end{equation*}
	Altogether we obtain
	\begin{equation}  \label{maingreen}
	\begin{aligned}
	\bE [G^*]^n[G]^{n-1} \langle \ul{G}\rangle = &\frac{1}{T} \Big(\bE [G^*]^n[G]^{n-1}  \langle \underline{G} \rangle^2-\bE [G^*]^n[G]^{n-1} \bE \langle \underline{G} \rangle^2+\frac{1}{N}\bE [G^*]^n[G]^{n-1}  \langle\underline{G^2} \rangle\\
	&\ \ +\frac{2n-2}{N^2}\bE[G^*]^n[G]^{n-2} \big(\underline{G^3}+2\ul{HG}\,mm' \big)-\bb EK-\sum_{k=2}^{l}\bb EL_{k}\\
	&\ \ +\frac{2n}{N^2}\bE[G^*]^{n-1}[G]^{n-1} \big(\underline{GG^{*2}}+2\ul{HG^*}\,\bar{m}\bar{m}'\big)-\frac{1}{N}\sum_{i,j}\bE \cal R^{(ji)}_{l+1}\Big)\,,
	\end{aligned}
	\end{equation}
	where $T\deq-z-2\bb E \ul{G}$. Note that by \eqref{mmm} and Theorem \ref{refthm1} we have
	\begin{equation} \label{bound for T}
	\frac{1}{T}=\frac{1}{-z-2m}+O_{\prec}\Big(\frac{1}{q}+\frac{1}{N\eta}\Big)=O_{\tau}(1)
	\end{equation}
	uniformly for all $z \in {\bf D}_{\tau}$. Let us look at the terms in \eqref{maingreen}. By \eqref{bound for H^2} we have
	\begin{equation} \label{4.11}
	\langle \ul{G} \rangle = [G]+O_{\prec}(N^{-1/2-\beta})\,.
	\end{equation}
	Together with Lemma \ref{prop4.4} (i) and H\"{o}lder's inequality we get
	\begin{equation} \label{4.9}
	\bE [G^*]^n[G]^{n-1}  \langle \underline{G} \rangle^2 \prec \Big(\frac{1}{q}+\frac{1}{N\eta}\Big)\, \cal M^{2n} + \Big(\frac{1}{q}+\frac{1}{N\eta}\Big) \frac{1}{\sqrt{N}q}\cal M^{2n-1}
	\end{equation}
	and
	\begin{equation}
    \bE [G^*]^n[G]^{n-1} \bE \langle \underline{G} \rangle^2 \prec \Big(\frac{1}{q}+\frac{1}{N\eta}\Big)\, \cal M^{2n} + \Big(\frac{1}{q}+\frac{1}{N\eta}\Big) \frac{1}{\sqrt{N}q}\cal M^{2n-1}  \,.
    \end{equation}
	Similarly, by Lemma \ref{prop4.4} (i), (iii) and H\"{o}lder's inequality we have
\begin{equation}
\frac{1}{N}\bE [G^*]^n[G]^{n-1}  \langle\underline{G^2} \rangle \prec  \Big(\frac{1}{q}+\frac{1}{N\eta}\Big) \frac{1}{N\eta}\, \cal M^{2n-1}\,.
\end{equation}
and
\begin{equation} \label{415}
\frac{2n-2}{N^2}\bE [G^*]^n[G]^{n-2}\underline{G^3} =\frac{2n-2}{N^2}\bE [G^*]^n[G]^{n-2}(\langle\underline{G^3}\rangle+\bb E \ul{G^3})\prec \Big(\frac{1}{N\eta}+\frac{1}{q}+\eta\Big)\Big(\frac{1}{N\eta}\Big)^2 \cal M^{2n-2}\,.
\end{equation}
Note that
\begin{equation} \label{416}
HG=I+zG \quad \mbox{and} \quad |z|\leq 6\,,
\end{equation} 
hence
\begin{equation} \label{4.17}
\frac{4n-4}{N^2}\bE [G^*]^n[G]^{n-2}\ul{HG}mm'=\frac{4n-4}{N^2}\bE [G^*]^n[G]^{n-2}(1+z\ul{G})mm' \prec \frac{1}{N^2}\cal M^{2n-2}\,.
\end{equation}
From resolvent identity, Theorem \ref{refthm1} and Lemma \ref{prop4.4} (iii) we have
\[
\bb E \ul{GG^{*2}}=\frac{\bb E\ul{G}-\bb E\ul{G^*}}{(2\ii \eta)^2}-\frac{\bb E \ul{G^{*2}}}{2\ii \eta}=-\frac{m-\bar{m}}{4\eta^2}+O_{\prec}\Big(\Big(\frac{1}{q}+\frac{1}{N\eta}+\eta\Big)\,\eta^{-2}\Big)\,.
\]
Thus
\begin{multline}
\frac{2n}{N^2}\bE[G^*]^{n-1}[G]^{n-1} \underline{GG^{*2}}=\frac{2n}{N^2}\bE[G^*]^{n-1}[G]^{n-1} (\langle\underline{GG^{*2}}\rangle+\bb E\underline{GG^{*2}})\\
=-\frac{n(m-m^*)}{2N^2\eta^2}\bb E |[G]|^{2n-2}+O_{\prec}\Big(\frac{1}{q}+\frac{1}{N\eta}+\eta\Big)\Big(\frac{1}{N\eta}\Big)^2 \cal M^{2n-2}\,.
\end{multline}
Similarly, by $N\cal C_2(H_{ii}) \asymp1$ and the differential rule \eqref{diff} one can easily check that
\begin{equation} \label{KKK}
\bE K \prec \frac{1}{N} \cal M^{2n-1}+ \frac{1}{N^2\eta} \cal M^{2n-2}\,.
\end{equation}
The estimate for the remainder term can be done routinely. One can follow, e.g.\,the proof of Lemma 3.4 (iii) in \cite{HKR}, and readily check that 
\begin{equation} \label{RRR}
\frac{1}{N}\sum_{i,j}\bE\cal R^{(ji)}_{l+1} \prec \frac{1}{N^{2n}}
\end{equation}
for $l$ large enough. From now on, we shall always assume the remainder term in cumulant expansion is negligible. Inserting the above estimates \eqref{bound for T}, \eqref{4.9} -- \eqref{415}, \eqref{4.17} -- \eqref{RRR} into \eqref{maingreen}, we have
\begin{equation} \label{4.16}
\begin{aligned}
\bE [G^*]^n[G]^{n-1} \langle \ul{G}\rangle =&\frac{n}{2N^2\eta^2}\bb E|[G]|^{2n-2}-\frac{1}{T}\sum_{k=2}^{l}\bE L_{k}\\&
+O_{\prec}(N^{-\delta})\cal M^{2n}+\sum_{r=1,2}O_{\prec}\Big(\frac{N^{-\xi}}{(N\eta)^r}+\Big(\frac{N^{-\delta}}{\sqrt{N}q}\Big)^r\Big)\cal M^{2n-r}\,,
\end{aligned}
\end{equation}
where we recall the definitions of $\delta,\xi$ from \eqref{gamma'},\eqref{chi}.
What is left, therefore, is the analysis of $\bE L_k$, $k \ge 2$.

\subsection{Abstract polynomials and  the recursive estimates} \label{sec4.2}
We now introduce some additional notations that will be used frequently in the analysis of $L_k$. To motivate them, we note that the proof relies on a calculus of products of expectations of random variables of the type $(G^{m})_{ij}$,  $(G^{*n})_{ij}$, $[{G}]$, $ [{G^*}]$ evaluated at $z \in \b D_{\tau}$, e.g.
\begin{equation*}
a(z,z^*) \, N^{3/2} \, \E [{G}] [{G^*}]\,, \qquad
a_{i_1 i_2 i_3 i_4}(z,z^*)\, N^{-1/2} \, \bb E (G^3)_{i_1i_2}G_{i_3i_4}\E [{G^*}]^2\,,
\end{equation*}
where $a(z, z^*)$ and $a_{i_1 i_2 i_3 i_4}(z,z^*)$ are uniformly bounded functions that may depend on $N$. It is convenient to classify such expressions depending on the exponent $\nu_1 \in \R$ and on the number $\nu_0$ of indices $i_k$. Below, we introduce the notations $\cal U^{(\nu_0,\nu_1)}(\cal Y)$, $\cal V^{(\nu_0,\nu_1)}(\cal Y)$ for the set of such expressions, where $\cal Y$ is the set of matrices appearing in them, in the above examples $\cal Y = \{G, G^*\}$.

To that end, we define a set of formal monomials in a set of formal variables. Here the word \textit{formal} refers to the fact that these definitions are purely algebraic and we do not assign any values to variables or monomials. The formal variables are constructed from a finite set of formal
matrices $\cal Y$ and the infinite set of formal indices $\{i_1,i_2,...\}$.

\begin{enumerate}
	\item[$\bullet$] For $\nu_0\in \bb N,\nu_1 \in \bb R$, denote by $\cal U^{(\nu_0,
		\nu_1)}(\cal Y)$ the set of monomials with coefficient $a_{i_1,...,i_{\nu_0}}N^{-\nu_1}$ in the variables $(Y^m)_{xy}$  and $[Y]$. Here $Y \in \cal Y$, $m\in \bb N_+$, $x,y \in \{i_1,...,i_{\nu_0}\}$, and $(a_{i_1,...,i_{\nu_0}})_{1\leq i_1,...,i_{\nu_0}\leq N}$ is some family of complex numbers that is uniformly bounded in $i_1,...,i_{\nu_0}$.
    \item[$\bullet$] Set $\cal U(\cal Y)=\bigcup_{\nu_0,\nu_1}\cal U^{(\nu_0,\nu_1)}(\cal Y)$.
\end{enumerate}
We also define the following subset of $\cal U(\cal Y)$.
\begin{enumerate}
	\item[$\bullet$] We denote by $\cal V^{(\nu_0,
		\nu_1)}(\cal Y)$ the subset of \,$ \cal U^{(\nu_0,
		\nu_1)}(\cal Y)$, where we further require $m \in \{1,2\}$ for all variables $(Y^m)_{xy}$.
	\item[$\bullet$] Set $\cal V(\cal Y)=\bigcup_{\nu_0,\nu_1}\cal V^{(\nu_0,\nu_1)}(\cal Y)\subset \cal U(\cal Y)$.
\end{enumerate}
Next, we define the following maps $\nu_0,\nu_1,\nu_2,\nu_3,\widetilde{\nu}_3,\nu_4$: $\cal U(\cal Y) \to \bb N$.
\begin{enumerate}
	\item For $U \in \cal U^{(\nu_0,\nu_1)}(\cal Y)$,  $(\nu_0(U),\nu_1(U))=(\nu_0,\nu_1)$.
	\item $\nu_2(U)=$ sum of $m-1$ of all $(Y^m)_{xy}$ in $U$ with $Y\in \cal Y$. 
	\item ${\nu}_3(U)=$ $2\wedge$(number of $(Y^m)_{xy}$ in $U$ with $x \ne y$ and $Y \in \cal Y$). Set $\widetilde{\nu}_3(U)=2-\nu_3(U)$.
	\item $\nu_4(U)=$ number of $ [Y]$ in $U$ with $Y \in \cal Y$.
\end{enumerate}

Next, we assign to each monomial $U \in  \cal U^{(\nu_0,
	\nu_1)}(\cal Y)$ a value $U_{i_1,...,i_{\nu_0}}$ as follows. Suppose that the set $\cal Y$ consists of $N \times N$ random matrices. Then for any $\nu_0$-tuple $(i_1,...,i_{\nu_0})\in \{1,2,...,N\}^{\nu_0}$ we
define the number $U_{i_1,...,i_{\nu_0}}$ as the one obtained by taking the formal expression $U$ and evaluating it with the laws of the matrices in $\cal Y$ and the numerical values of  $i_1,...,i_{\nu_0}$. In the following arguments, the set $\cal Y$ will consist of Green functions of $H$ for
the spectral parameter $z$ defined in \eqref{z}, and the indices $i_1,...,i_{\nu_0}$ will be summed over.

The next result is a straightforward consequence of Lemma \ref{prop4.4} (i), (iii) and H\"{o}lder's inequality whose proof we omit.
\begin{lemma} \label{lem4.2}
	Let $\cal Y=\{G,G^*\}$, and fix $U \in \cal U^{(\nu_0, \nu_1)}(\cal Y)$. Then
	\begin{equation} \label{417}
	\sum_{i_1,...,i_{\nu_0}}\bb E  U_{i_1,...,i_{\nu_0}} =O_{\prec}(N^{b_0(U)}) \cdot \cal M^{\nu_4(U)}\,,
	\end{equation}
	where $b_0(U)=\nu_0(U)-\nu_1(U)+\alpha\nu_2(U)+(\alpha/2-1/2){\nu}_3(U)$.
\end{lemma}
Our first estimate is the following improved bound for the LHS of \eqref{417}, whose proof is postponed to Section \ref{sec5.1}. The necessity of this result is explained in Remark \ref{rmk4.7} below.
\begin{lemma} \label{lem4.3}
	Let us adopt the assumptions in Lemma \ref{lem4.2}. Let $V \in \cal V^{(\nu_0, \nu_1)}(\cal Y)$ satisfying $\nu_2(V)\geq 1$, we have
	\begin{equation*} 
	\sum_{i_1,...,i_{\nu_0}} \bb E  V_{i_1,...,i_{\nu_0}} \prec\cal B(V)\,,
	\end{equation*}
	where 
	\begin{equation*}
	\cal B(V)=N^{b_0(V)}  \cal M^{\nu_4(V)+1}+\sum_{k=0}^{\nu_4(V)}N^{b_0(V)}\frac{1}{(N\eta)^k} \Big(\eta+\frac{1}{(N\eta)^{\widetilde{\nu}_3(V)/2}}\Big)\cdot \cal M^{\nu_4(V)-k}\,.
	\end{equation*}
\end{lemma}
\begin{example}
	Let $\cal Y=\{G,G^*\}$ and set
	\[
	U\equiv U_{ij}\deq \frac{1}{N^2} \cal C_4(H_{ij})  \bb E [G]^{n-1}[G^*]^{n-1} (G^{*2})_{ii}G^*_{jj}G_{ii}G_{jj}\,.
	\]
	Note that $Nq^2\cal C_4(H_{ij})\asymp 1$, and thus $U \in \cal V^{(\nu_0,\nu_1)}(\cal Y)\subset \cal U^{(\nu_0,\nu_1)}(\cal Y)$, with $\nu_0=2$ and $\nu_1=2+(1+2\log_N q)=3+2\beta$. We also have $\nu_2(U)=2-1=1$, $\nu_3(U)=2\wedge 0=0$, $\widetilde{\nu}_3(U)=2-0=2$, $\nu_4(U)=2n-2$, and $b_0(U)=-1-2\beta+\alpha$.
	
	By Lemma \ref{prop4.4}, we have
	\[
	\sum_{i,j} U_{ij} \prec \frac{1}{N^3q^2} \cdot N^{2} \cdot \cal M^{2n-2} \cdot \frac{1}{\eta} =\frac{1}{Nq^2\eta}\cdot \cal M^{2n-2}=N^{b_0(U)}\cdot \cal M^{2n-2},
	\]
	which agrees with Lemma \ref{lem4.2}. On the other hand, Lemma \ref{lem4.3} implies the improved estimate
	\[
	\sum_{i,j} U_{ij} \prec \frac{1}{Nq^2\eta}\cal M^{2n-1}+\sum_{k=0}^{2n-2} \frac{1}{Nq^2\eta}\frac{1}{(N\eta)^k}\Big(\eta+\frac{1}{N\eta}\Big)\cdot \cal M^{2n-2-k}\,.
	\]
\end{example}
In order to handle all terms in $L_k$, we also need the following formal polynomials.
\begin{enumerate}

\item[$\bullet$] For $\nu_0 \in \bb N,\nu_1 \in \bb R$, denote by $\cal W^{(\nu_0,\nu_1)}(\cal Y)$ the set of monomials with coefficient $a_{i_1,...,i_{\nu_0}}N^{-\nu_1}$ in the variable $[Y]$ and also contain exactly one factor of $\langle Y^{(1)}_{x_1y_1}\cdots Y^{(k)}_{x_ky_k}\rangle$. Here $Y,Y^{(1)},...,Y^{(k)} \in \cal Y$, $k\in \bb N_+$, $x,y \in \{i_1,...,i_{\nu_0}\}$, and $(a_{i_1,...,i_{\nu_0}})_{1\leq i_1,...,i_{\nu_0}\leq N}$ is some family of complex numbers that is uniformly bounded in $i_1,...,i_{\nu_0}$.

\item[$\bullet$] Set $\cal W(\cal Y)=\bigcup_{\nu_0,\nu_1}\cal W^{(\nu_0,\nu_1)}(\cal Y)$.
\end{enumerate}
Next, we define the following maps $\nu_0,\nu_1,\nu_3,\widetilde{\nu}_3,\nu_4$: $\cal W(\cal Y) \to \bb N$.
\begin{enumerate}
	\item For $W \in \cal W^{(\nu_0,\nu_1)}(\cal Y)$,  $(\nu_0(W),\nu_1(W))=(\nu_0,\nu_1)$.
	\item ${\nu}_3(W)=$ $2\wedge $(number of $Y_{xy}$ in $W$ with $x \ne y$ and $Y \in \cal Y$). Set $\widetilde{\nu}_3(W)=2-\nu_3(W) $.
	\item $\nu_4(W)=$ number of $ [Y]$ in $W$ with $Y \in \cal Y$.
\end{enumerate}
The following is a trivial result from Lemma \ref{Ward}.
\begin{lemma} 
	\label{lem4.4}
	Let $\cal Y=\{G,G^*\}$, and fix $W \in \cal W^{(\nu_0, \nu_1)}(\cal Y)$. Then
	\begin{equation*} 
	\sum_{i_1,...,i_{\nu_0}} \bb E  W_{i_1,...,i_{\nu_0}} =O_{\prec}(N^{b_1(W)}) \cdot \cal M^{\nu_4(W)}\,,
	\end{equation*}
	where $b_1(W)=\nu_0(W)-\nu_1(W)+(\alpha/2-1/2){\nu}_3(W)$.
\end{lemma}
We have the following improved estimate for Lemma \ref{lem4.4}, whose proof is postponed to Section \ref{sec5.2}. The necessity of this result is explained in Remark \ref{rmk4.7} below.
\begin{lemma} \label{lem4.5}
	 Let us adopt the assumptions in Lemma \ref{lem4.4}. We have
\begin{equation*} 
\sum_{i_1,...,i_{\nu_0}}\bb E  W_{i_1,...,i_{\nu_0}} \prec\cal B^{(1)}(W)\,,
\end{equation*}
where 
\begin{multline*}
\cal B^{(1)}(W)=N^{b_1(W)}\Big(\frac{1}{(N\eta)^{\widetilde{\nu}_3(W)/2}}+\frac{1}{\sqrt{N}q}+\cal M\Big)\cdot \cal M^{\nu_4(W)}\\
+\sum_{k=1}^{\nu_4(V)}N^{b_1(W)}\frac{1}{(N\eta)^{k}}\Big(\frac{\eta}{q}+\frac{1}{(N\eta)^{\widetilde{\nu}_3(W)/2}}\Big) \cdot \cal M^{\nu_4(W)-k}\,.
\end{multline*}
\end{lemma}

We close this section with the following estimate.
	\begin{lemma} \label{lemH}
		Let $\cal Y=\{G,G^*\}$, and fix $U \in \cal U^{(\nu_0, \nu_1)}(\cal Y)$. For $i,j \in \{i_1,...i_{\nu_0}\}$, we have
		\begin{equation} \label{41777}
		\sum_{i_1,...,i_{\nu_0}}\bb E H_{ij}   U_{i_1,...,i_{\nu_0}} =O_{\prec}(N^{b_2(U)}) \cdot \sum_{k=0}^{\nu_4(U)}\frac{1}{(N\eta)^k}\cal M^{\nu_4(U)-k}\,,
		\end{equation}
		where $b_2(U)=\nu_0(U)-\nu_1(U)+\alpha\nu_2(U)-1$.
	\end{lemma}
\begin{proof}
	The proof follows by applying Lemma \ref{lem:cumulant_expansion} on LHS of \eqref{41777} with $h=H_{ij}$, and then estimating the result by Lemma \ref{lem4.2}. We omit the details here.
\end{proof}
In the next two section we shall estimate $\bb E L_{2}$ and $\bb E L_3$ using the above lemmas.
\subsection{The estimate of $\bE L_2$} \label{section4.3}
In this section we prove the following result.
\begin{lemma} \label{lem4.6}
Let $L_2$ be as in \eqref{4.16}. Let $\delta,\xi$ be as in \eqref{gamma'}, \eqref{chi}. We have 
\begin{equation} \label{L_k}
\bb E L_2\prec N^{-\delta}\cal M^{2n}+\sum_{r=1}^{2n}\bigg(\frac{N^{-\xi}}{(N\eta)^{r}}+\bigg(\frac{N^{-\delta/4}}{\sqrt{N}q}\bigg)^{r}\bigg)\cal M^{2n-r}\,.
\end{equation}
\end{lemma}
\begin{proof}
The differential $\partial H^2_{ij}$ gives rise to terms of three types depending on how many derivatives act on $G_{ij}$. We deal with each type separately.

Step 1. Let us look at the case when both derivatives in $L_2$ act on $G_{ij}$, namely the term
\[
\bb E L_{2,1}\deq N^{-1} \cdot \sum\limits_{i,j}\left[\frac{1}{2!}\mathcal{C}_{3}(H_{ji})\bE\langle [G^*]^n [G]^{n-1} \rangle\frac{\partial^2  G_{ij}}{\partial H_{ji}^2}\right]\,.
\]
By Lemma \ref{Tlemh} and the identity $\bb E \langle X \rangle Y=\bb E X \langle Y \rangle$, we see that the worst term in $L_{2,1}$ is
\begin{equation} \label{4.41}
\frac{1}{N^2q}\sum_{i,j}C_{ij} \bb E  [G^*]^n [G]^{n-1}  \langle G_{ii}G_{jj}G_{ij} \rangle \eqd \sum_{i,j} \bb E W_{i,j} \,,
\end{equation}
where $C_{ij}$ are constants uniformly bounded in $i,j$, and  $W \in \cal W(\{G,G^*\})$. Note that $\nu_0(W)=2$, $\nu_1(W)=2+\beta$, ${\nu_3}(W)=1$, $\nu_4(W)=2n-1$, and $b_1(W)=-\beta -(1/2-\alpha/2)$. Thus Lemma \ref{lem4.5} shows
\begin{multline*}
\sum_{i,j} \bb E W_{ij} \prec \frac{1}{\sqrt{N\eta}q} \cal M^{2n}+\frac{1}{\sqrt{N\eta}q}\Big(\frac{1}{\sqrt{N\eta}}+\frac{1}{\sqrt{N}q}\Big)\cdot \cal M^{2n-1}
+\sum_{k=1}^{2n-1} \frac{1}{\sqrt{N\eta}q}\frac{1}{(N\eta)^k}\Big(\frac{\eta}{q}+\frac{1}{\sqrt{N\eta}}\Big)\cdot \cal M^{2n-1-k}\\
\prec N^{-\beta} \cal M^{2n}+N^{-\beta}\Big(\frac{1}{N\eta}+\frac{1}{\sqrt{N}q}\Big)\cdot \cal M^{2n-1}+\sum_{r=2}^{2n}\Big(\frac{N^{-\beta}}{\sqrt{N}q}\frac{N^{-\alpha/2}}{(N\eta)^{r-1}}+\frac{N^{-\beta}}{(N\eta)^{r}}\Big) \cal M^{2n-r}\,,
\end{multline*}
which is bounded by the RHS of \eqref{L_k}. Similarly, one can show that the other terms in $L_{2,1}$ satisfy the same bound.

Step 2. Let us look at the case when only one derivative in $L_2$ acts on $G_{ij}$, namely the term
\begin{equation} \label{L22}
\bb E L_{2,2}\deq N^{-1} \cdot \sum\limits_{i,j}\left[\mathcal{C}_{3}(H_{ji})\bE \frac{\partial \langle [G^*]^n [G]^{n-1} \rangle}{\partial H_{ji}}\frac{\partial  G_{ij}}{\partial H_{ji}}\right]\,.
\end{equation}
By \eqref{diff}, we see that the worst terms above will contain no off-diagonal terms of $G$ from the second differential. Let us pick a representative of these, which is
\begin{multline} \label{4.43}
\frac{1}{N^3q}\sum_{i,j}C_{ij} \bb E  [G^*]^{n-1} [G]^{n-1} ((G^{*2})_{ij}+2H_{ij}) G_{ii}G_{jj}\\
=\frac{1}{N^3q}\sum_{i,j}C_{ij} \bb E  [G^*]^{n-1} [G]^{n-1} (G^{*2})_{ij} G_{ii}G_{jj}+O_{\prec}\Big(\frac{1}{N^2q}\Big)\sum_{k=0}^{2n-2}\frac{1}{(N\eta)^k}\cal M^{2n-2-k}\\
\eqd\frac{1}{N^3q}\sum_{i,j} \bb E V_{ij}+O_{\prec}\Big(\frac{1}{N^2q}\Big)\sum_{k=0}^{2n-2}\frac{1}{(N\eta)^k}\cal M^{2n-2-k}\,.
\end{multline}
Here $C_{ij}$ are constants uniformly bounded in $i,j$, and in the first step of \eqref{4.43} we used Lemma \ref{lemH}. Note that $V \in \cal V$ satisfies $\nu_0(V)=2$, $\nu_1(V)=3+\beta$, $\nu_2(V)=1$, ${\nu_3}(V)=1$, $\nu_4(V)=2n-2$, and $b_0(V)=3(\alpha-1)/2-\beta$. Thus Lemma \ref{lem4.3} shows
\begin{multline} \label{4.56}
\frac{1}{N^3q}\sum_{i,j} \bb E V_{ij} \prec \frac{1}{(N\eta)^{3/2}q} \cal M^{2n-1}+\sum_{k=0}^{2n-2}\frac{1}{(N\eta)^{3/2}q}\Big(\frac{1}{N\eta}\Big)^k \Big(\eta+\frac{1}{\sqrt{N\eta}}\Big)\cdot \cal M^{2n-2-k}\\
\prec \frac{N^{-\beta}}{N\eta}\cal M^{2n-1}+\sum_{r=2}^{2n}\Big(\frac{1}{(N\eta)^{r-3/2}Nq}+\frac{N^{-\beta}}{(N\eta)^r}\Big)\cal M^{2n-r}\,.
\end{multline}
By \eqref{4.56} and
\[
\frac{1}{(N\eta)^{r-3/2}Nq} =\frac{1}{(N\eta)^{r-3/2}(\sqrt{N}q)^{3/2}} \Big(\frac{q^2}{N}\Big)^{1/4} \prec \bigg(\frac{N^{-\delta/8}}{(N\eta)^r}+\frac{N^{-\delta/4}}{\sqrt{N}q}\bigg)^r\,,
\]
we see that \eqref{4.43} is bounded by the RHS of \eqref{L_k}. Similarly, the other terms in $L_{2,1}$ can be shown to satisfy the same bound.

Step 3.  Let us look at the case when no derivatives in $L_2$ act on $G_{ij}$, namely the term
\[
\bb EL_{2,3}\deq N^{-1} \cdot \sum\limits_{i,j}\left[\frac{1}{2!}\mathcal{C}_{3}(H_{ji})\bE \frac{\partial^2 \langle [G^*]^n [G]^{n-1} \rangle}{\partial H^2_{ji}}G_{ij}\right]\,.
\]
Similar as in Step 2, one can use Lemma \ref{lem4.3} to show that
\[
\bb E L_{2,3}  \prec\sum_{r=1}^{2n}\bigg(\frac{N^{-\xi}}{(N\eta)^{r}}+\bigg(\frac{N^{-\delta/4}}{\sqrt{N}q}\bigg)^{r}\bigg)\cal M^{2n-r}\,.
\]
From Steps 1--3 we conclude the proof.
\end{proof}

\subsection{The estimate of $\bE L_3$} \label{sec4.4}
 Now let us look at the case $k=3$. This is the crucial case where we see the cancellation between $\langle\ul{G}\rangle$ and $(\ul{H^2}-1)mm'$.
 We shall prove the following lemma.
 \begin{lemma} \label{lem4.7}
 	Let $L_3$ be as in \eqref{4.16}. Let $\delta,\xi$ be as in \eqref{gamma'}, \eqref{chi}. We have 
 	\begin{equation} \label{L_3}
 	\bb E L_3\prec N^{-\delta}\cal M^{2n}+\sum_{r=1}^{2n}\bigg(\frac{N^{-\xi}}{(N\eta)^{r}}+\bigg(\frac{N^{-\delta/4}}{\sqrt{N}q}\bigg)^{r}\bigg)\cal M^{2n-r}\,.
 	\end{equation}
 \end{lemma} 
\begin{proof} We still split the estimates basing on how many derivatives hit $G_{ij}$.
	
Step 1. We investigate the case when all derivatives in $L_3$ act on $G_{ij}$, namely the term
\[
\bb EL_{3,1}\deq N^{-1} \cdot \sum\limits_{i,j}\left[\frac{1}{3!}\mathcal{C}_{4}(H_{ji})\bE\langle [G^*]^n [G]^{n-1} \rangle\frac{\partial^3  G_{ij}}{\partial H_{ji}^3}\right]\,.
\]
From Lemma \ref{Tlemh} we see that the worst term in $L_{3,1}$ is 
\begin{equation*} 
\frac{1}{N^2q^2}\sum_{i,j}C_{ij} \bb E  [G^*]^n [G]^{n-1}  \langle G^2_{ii}G^2_{jj}\rangle\eqd \frac{1}{N^2q^2} \sum_{i,j}\bb E W_{ij}\,,
\end{equation*}
where $W \in \cal W$. Note that $\nu_0(W)=2$, $\nu_1(W)=2+2\beta$, ${\nu_3}(W)=0$, $\nu_4(W)=2n-1$, and $b_1(W)=-2\beta $. Thus Lemma \ref{lem4.5} shows
\begin{multline*}
\frac{1}{N^2q^2} \sum_{i,j}\bb E W_{ij} \prec \frac{1}{q^2} \cal M^{2n}+\frac{1}{q^2}\Big(\frac{1}{N\eta}+\frac{1}{\sqrt{N}q}\Big)\cdot\cal M^{2n-1}
+\sum_{k=1}^{2n-1}\frac{1}{q^2}\frac{1}{(N\eta)^k}\Big(\frac{\eta}{q}+\frac{1}{N\eta}\Big) \cdot \cal M^{2n-1-k}\\
\prec N^{-2\beta}\cdot\cal M^{2n}+N^{-2\beta}\Big(\frac{1}{N\eta}+\frac{1}{\sqrt{N}q}\Big)\cdot\cal M^{2n-1}+ \sum_{r=2}^{2n} O_{\prec}\bigg(\frac{N^{-\beta/2}}{Nq^2}\frac{N^{-\beta/2}}{(N\eta)^{r-2}}+\frac{N^{-2\beta}}{(N\eta)^r}\bigg)\cdot \cal M^{2n-r}\,,
\end{multline*}
which is bounded by RHS of \eqref{L_3}.

Step 2. Let us look at the case when only one derivative in $L_3$ acts on $G_{ij}$, namely the term
\[
\bb E L_{3,2}\deq \frac{1}{2N} \cdot \sum\limits_{i,j}\left[\mathcal{C}_{4}(H_{ji})\bE \frac{\partial^2 \langle [G^*]^n [G]^{n-1} \rangle}{\partial H^2_{ji}}\frac{\partial  G_{ij}}{\partial H_{ji}}\right]\,.
\]
We see that one of the worst terms is 
\begin{multline}  \label{4.49}
-\frac{n}{2N^3q^2}\sum_{i,j}s_4(1+C_i\delta_{ij}) \bb E  [G^*]^{n-1} [G]^{n-1} (4(G^{*2})_{ii}G^*_{jj}-4\bar{m}\bar{m}') G_{ii}G_{jj}\\
=-\frac{n}{2N^3q^2}\sum_{i,j}s_4 \bb E  [G^*]^{n-1} [G]^{n-1} (4(G^{*2})_{ii}G^*_{jj}-4\bar{m}\bar{m}') G_{ii}G_{jj}+O_{\prec} \Big(\frac{1}{N\eta}\frac{1}{Nq^2}\Big)\cal M^{2n-2}
\end{multline}
where $s_4=Nq^2\cal C_4(H_{12})\asymp 1$, and $C_i$ are constants uniformly bounded in $i$. Now let us look at the first term on RHS of \eqref{4.49}, which is
\begin{equation*} 
-\frac{2n}{N^3q^2}\sum_{i,j}s_4 \bb E  [G^*]^{n-1} [G]^{n-1} (G^{*2})_{ii}G^*_{jj} G_{ii}G_{jj} \eqd \sum_{i,j}\bb E V_{ij}\,,
\end{equation*}
where $V \in \cal V(\{G,G^*\})$. By the resolvent identity $\bar{z}G^*=HG^*-I$ and Lemma \ref{lem:cumulant_expansion}, we have
\begin{multline*} 
\sum_{i,j}\bb E V_{ij}=\frac{1}{\bar{z}+\bb E \ul{G^*}} \frac{2s_4n}{N^3q^2} \sum_{i,j}\bb E\bigg[[G^*]^{n-1} [G]^{n-1} G^*_{jj} G_{ii}G_{jj}\cdot\big(G^*_{ii}\ul{G^{*2}}+(G^{*2})_{ii}\langle \ul{G^*}\rangle+2N^{-1}(G^{*3})_{ii}\big) \\
+\frac{1}{N} [G^*]^{n-1} [G]^{n-1}\cdot \big(2(G^{*3})_{ij}G^*_{ij}G_{ii}G_{jj}+2(G^{*2}G)_{ii}G^{*}_{jj}G_{ii}G_{jj}+2(G^{*2}G)_{ij}G^*_{jj}G_{ii}G_{ij}\big)\\
+[G^*]^{n-1} [G]^{n-1} G^*_{ii}G^*_{jj} G_{ii}G_{jj} -\widetilde{K}-\sum_i\sum_{k=2}^l \widetilde{L}_{k} - \sum_{a=1}^N\widetilde{\cal R}^{(a)}_{l+1}\bigg]\,,
\end{multline*}
where $\widetilde{K}$, $\widetilde{L}_k$, and $\widetilde{\cal R}^{(a)}_{l+1}$ are defined similar as $K,L_k,\cal R^{(ij)}_{l+1}$ in \eqref{4.6}. By \eqref{diffH} and Lemmas \ref{lem4.2},\ref{lem4.3}, one can check that
\begin{multline} \label{439}
\sum_{i,j}\bb E V_{ij}=\frac{1}{\bar{z}+\bb E \ul{G^*}} \frac{2s_4n}{N^3q^2} \sum_{i,j}\bb E\Big([G^*]^{n-1} [G]^{n-1} G^*_{jj} G_{ii}G_{jj} G^*_{ii}\ul{G^{*2}}+[G^*]^{n-1} [G]^{n-1} G^*_{ii}G^*_{jj} G_{ii}G_{jj}\Big)\\
+\sum_{r=1}^{2n}O_{\prec}\bigg(\frac{N^{-\xi}}{(N\eta)^{r}}+\bigg(\frac{N^{-\delta/4}}{\sqrt{N}q}\bigg)^{r}\bigg)\cal M^{2n-r}\,.
\end{multline}
Similarly, for the first term on RHS of \eqref{439}, we have
\begin{multline} \label{440}
\frac{1}{\bar{z}+\bb E \ul{G^*}} \frac{2s_4n}{N^3q^2} \sum_{i,j}\bb E[G^*]^{n-1} [G]^{n-1} G^*_{jj} G_{ii}G_{jj}G^*_{ii}\ul{G^{*2}}\\=\frac{1}{-\bar{z}-2\bb E\ul{G^*}}\frac{1}{\bar{z}+\bb E \ul{G^*}} \frac{2s_4n}{N^3q^2} \sum_{i,j}\bb E[G^*]^{n-1} [G]^{n-1} G^*_{jj} G_{ii}G_{jj}G^*_{ii}\ul{G^{*}}\\
+\sum_{r=1}^{2n}O_{\prec}\bigg(\frac{N^{-\xi}}{(N\eta)^{r}}+\bigg(\frac{N^{-\delta/4}}{\sqrt{N}q}\bigg)^{r}\bigg)\cal M^{2n-r}\,.
\end{multline}
By \eqref{439},\eqref{440} and Theorem \ref{refthm1}, we have
\begin{multline} \label{441}
\sum_{i,j}\bb E V_{ij}=\frac{2s_4n}{N^3q^2} \frac{1}{\bar{z}+\bar{m}}\Big( \frac{\bar{m}^3}{-\bar{z}-2\bar{m}}+\bar{m}^2\Big)\sum_{i,j}\bb E[G^*]^{n-1} [G]^{n-1}  G_{ii}G_{jj}\\
+\sum_{r=1}^{2n}O_{\prec}\bigg(\frac{N^{-\xi}}{(N\eta)^{r}}+\bigg(\frac{N^{-\delta/4}}{\sqrt{N}q}\bigg)^{r}\bigg)\cal M^{2n-r}\,.
\end{multline}
By $\bar{m}/(-\bar{z}-2\bar{m})=\bar{m}'$ we have
\begin{equation} \label{4,43}
\frac{1}{\bar{z}+\bar{m}}\Big(  \frac{\bar{m}^3}{-\bar{z}-2\bar{m}}+\bar{m}^2\Big)=-\bar{m}\bar{m}'\,.
\end{equation}
Combining \eqref{4.49}, \eqref{441} and \eqref{4,43}, we see the crucial cancellation of the first two terms on RHS of \eqref{4.49}. As a result, we obtain
\[
\eqref{4.49} \prec \sum_{r=1}^{2n}\bigg(\frac{N^{-\xi}}{(N\eta)^{r}}+\bigg(\frac{N^{-\delta/4}}{\sqrt{N}q}\bigg)^{r}\bigg)\cal M^{2n-r}
\]
as desired. The other terms in $\bb E L_{3,2}$ can be directly estimated by \eqref{diff} and Lemma \ref{lem4.3}, and one readily checks that they satisfy the bound on RHS of \eqref{L_3}.

Step 3. The remaining two cases, i.e. when two derivatives or no derivative act on $G_{ij}$, can be analyzed similarly using \eqref{diff} and Lemma \ref{lem4.3}. Note that the estimate is easier than those in Steps 1 and 2: by \eqref{diff}, every term now contains either at least two off diagonal entries of the Green function or derivatives of $\ul{H^2}-1$. We omit the details.

From Steps 1--3 we conclude the proof of Lemma \ref{lem4.7}.
\end{proof}

\subsection{Put things together}  \label{sec4.5}
Up to now, what is left, is the estimate of  $L_k$ for $k \geq 4$. This is similar but easier than the cases when $k=2,3$. In fact, by Lemma \ref{Tlemh} we see that there will be additional factors of $1/q$ in $L_k$ when $k \ge 4$. By a direct estimate using \eqref{diff} and Lemma \ref{lem4.3}, we have
\begin{equation} \label{4.42}
\sum_{k=4}^{l}\bb E L_k \prec  N^{-\delta}\cal M^{2n}+\sum_{r=1}^{2n}\bigg(\frac{N^{-\xi}}{(N\eta)^{r}}+\bigg(\frac{N^{-\delta/4}}{\sqrt{N}q}\bigg)^{r}\bigg)\cal M^{2n-r}
\end{equation}
for any fixed $l \in \bb N_+$. 

By \eqref{4.16}, \eqref{4.42}, Lemmas \ref{lem4.6} and \ref{lem4.7}, we conclude the proof of \eqref{est1}.

\begin{remark} \label{rmk4.7}
	The input of Lemmas \ref{lem4.3} and \ref{lem4.5} are essential in finishing the proof. For example, in the estimate of \eqref{4.41}, the bound from Lemma \ref{lem4.2} only implies
	\[
	\eqref{4.41} \prec  \frac{1}{N^2q} \cdot N^2 \cdot \frac{1}{\sqrt{N\eta}} \cdot \cal M^{2n-1}=\frac{1}{\sqrt{N\eta}q} \cal M^{2n-1}\,,
	\]
	which is not enough to deduce \eqref{L_k}. Also, in the estimate of \eqref{4.43}, the trivial bound from Lemma \ref{lem4.4} only implies
	\[
	\frac{1}{N^3q}\sum_{i,j} \bb E V_{ij} \prec \frac{1}{(N\eta)^{3/2}q} \cal M^{2n-2}\,,
	\]
	which is not enough to deduce \eqref{L_k}.
\end{remark}

\subsection{Proof of \eqref{est2}} \label{sec4.6}
The proof of \eqref{est2} is similar to that of \eqref{est1}. We only sketch the main steps.

Step 1. By Lemma \ref{lem:cumulant_expansion}, we have
\begin{multline} \label{4.15}
\bE [G^*]^n[G]^{n-1} (\ul{H^2}-1)=\frac{1}{N}\sum_{i,j}\bE H_{ij}^2 [G^*]^n[G]^{n-1}-\bb E[G^*]^{n}[G]^{n-1}\\
= \bb E[G^*]^{n}[G]^{n-1}-\frac{n}{N}\bb E [G^*]^{n-1}[G]^{n-1}(2\ul{HG^2}+4\ul{H^2}mm') -\frac{n-1}{N}\bb E [G^*]^{n}[G]^{n-2}(2\ul{HG^{*2}}+4\ul{H^2}\bar{m}\bar{m}')\\
+\bE\widehat{K}+\sum_{k=2}^l\bE\widehat{L}_{k,1} +\sum_{k=2}^l\bE\widehat{L}_{k,2}+\sum_{i,j}\widehat{\cal R}^{(ij)}_{l+1} -\bb E[G^*]^{n}[G]^{n-1}\,,
\end{multline}
where
\begin{equation*} 
\widehat{K}=N^{-2} \sum\limits_{i}\frac{\partial (H_{ii} [G^*]^n[G]^{n-1})}{\partial H_{ii}}(N\cal C_2(H_{ii})-2)\,,
\end{equation*}
\begin{equation*} 
\widehat{L}_{k,1}=N^{-1} \cdot \sum\limits_{i,j}\left(\frac{1}{k!}\mathcal{C}_{k+1}(H_{ji})H_{ij}\frac{\partial^k(  [G^*]^n[G]^{n-1})}{\partial H_{ji}^k}\right)\,, 
\end{equation*}
and
\begin{equation*}
\widehat{L}_{k,2}=N^{-1} \cdot \sum\limits_{i,j}\left(\frac{1}{(k-1)!}\mathcal{C}_{k+1}(H_{ji})\frac{\partial^{k-1}(  [G^*]^n[G]^{n-1})}{\partial H_{ji}^{k-1}}\right)\,.
\end{equation*}
Here $l$ is a fixed positive integer to be chosen later, and $\widehat{\cal R}_{l+1}^{(ji)}$ is a remainder term defined analogously to $\cal R_{l+1}$ in (\ref{eq:cumulant_expansion}). Notice the cancellation between the first and last terms on RHS of \eqref{4.15}. By \eqref{bound for H^2} and \eqref{416} we have
\begin{equation} \label{aaaaaa}
\bE [G^*]^n[G]^{n-1} (\ul{H^2}-1)=O_{\prec}\Big(\frac{1}{N}\Big)\cal M^{2n-1}+O_{\prec}\Big(\frac{1}{N^2\eta}\Big)\cal M^{2n-2}+\sum_{k=2}^l\bE\widehat{L}_{k,1} +\sum_{k=2}^l\bE\widehat{L}_{k,2}
\end{equation}

Step 2. Let us estimate $\sum_{k=2}^l\bE\widehat{L}_{k,1}$. By Lemma \ref{lem:cumulant_expansion} we have
\[
\sum_{k=2}^l\bE\widehat{L}_{k,1}=\sum_{k=2}^l \sum_{s=1}^{l'}N^{-1} \cdot \sum\limits_{i,j}\left(\frac{1}{k!s!}\mathcal{C}_{k+1}(H_{ji})\cal C_{s+1}(H_{ij})
\bb E\frac{\partial^{k+s}(  [G^*]^n[G]^{n-1})}{\partial H_{ji}^{k+s}}\right)\,,
\]
assuming the remainder term is small enough for large $l'$. By Lemma \ref{Tlemh} and $$\frac{\partial^m [G]}{\partial H_{ij}^m}\prec \frac{1}{N\eta}\,,$$
we have
\begin{equation*} 
\sum_{k=2}^l\bE\widehat{L}_{k,1}\prec \sum_{r=2}^{2n}\frac{1}{Nq}\frac{1}{(N\eta)^{r-1}}\cal M^{2n-r} \prec \sum_{r=2}^{2n-r}\bigg(\frac{N^{-1/4}}{(N\eta)^r}+\Big(\frac{N^{-1/4}}{\sqrt{N}q}\Big)^r\bigg) \cal M^{2n-r}\,.
\end{equation*}

Step 3. Now let us estimate at $\sum_{k=2}^l\bE\widehat{L}_{k,2}$. Note that in Section \ref{section4.3}, we have estimated $\bE L_{2,2}$ defined in \eqref{L22}. In particular, we have estimated the term
\[
N^{-1}\sum\limits_{i,j}\left[\mathcal{C}_{3}(H_{ji})\bE \frac{\partial \langle [G^*]^n [G]^{n-1} \rangle}{\partial H_{ji}}G_{ii}G_{jj}\right]\,,
\]
where the method used can be applied almost exactly in estimating $\bE \widehat{L}_{2,2}$. Similarly, we have estimated  
\[
\frac{1}{2N}\sum\limits_{i,j}\left[\mathcal{C}_{4}(H_{ji})\bE \frac{\partial^2 \langle [G^*]^n [G]^{n-1} \rangle}{\partial H^2_{ji}}G_{ii}G_{jj}\right]
\]
in Section \ref{sec4.4}, and method can be applied in estimating $\bE \widehat{L}_{3,2}$. Additionally, we can also estimate $\bE\widehat{L}_{k,2},k\geq 4$ using Lemmas \ref{lem4.3} and \ref{lemH}. One can check that
\begin{equation} \label{4.52}
\sum_{k=2}^l\bE\widehat{L}_{k,2}\prec\sum_{r=1}^{2n}\bigg(\frac{N^{-\xi}}{(N\eta)^{r}}+\bigg(\frac{N^{-\delta/4}}{\sqrt{N}q}\bigg)^{r}\bigg)\cal M^{2n-r}\,.
\end{equation}

Step 4. Combining \eqref{aaaaaa}--\eqref{4.52} we conclude the proof of \eqref{est2}.

\section{Estimates of general polynomials of Green functions} \label{sec5}
In this section we prove Lemmas \ref{lem4.3} and \ref{lem4.5}.
\subsection{Proof of Lemma \ref{lem4.3}} \label{sec5.1}
	To simplify notations, we shall prove the lemma for $\cal Y=\{G\}$, and one easily checks that the proof is the same for $\cal Y=\{G,G^{*}\}$. Let us take a general term  $V \in \cal V^{(\nu_0,\nu_1)}(\{G\})$, and consider
	\begin{equation}  \label{420}
	\sum_{i_1,...,i_{\nu_0}} \bb E  V_{i_1,...,i_{\nu_0}} =\frac{1}{N^{\nu_1}}\sum_{i_1,...,i_{\nu_0}}a_{i_1,...,i_{\nu_0}} \bb E G_{x_1y_1}\cdots G_{x_ky_k}(G^2)_{z_1w_1}\cdots (G^2)_{z_{\nu_2}w_{\nu_2}} [G] ^{\nu_4}\,,
	\end{equation}
	where $x_1,y_1,...,x_k,y_k, z_1, w_1,...,z_{\nu_2},w_{\nu_2} \in \{i_1,...,i_{\nu_0}\}$, and $a_{i_1,...,i_{\nu_0}}$ are complex numbers uniformly bounded in $i_1,...,i_{\nu_0}$. We break the proof into four steps.
	
	Step 1. Since $\nu_2\geq 1$, we can use resolvent identity $z(G^2)_{z_1w_1}=(HG^2)_{z_1w_1}-G_{z_1w_1}$ and Lemma \ref{lem:cumulant_expansion} to get
	\begin{multline} \label{421}
	\bb E  V_{i_1,...,i_{\nu_0}}=\frac{1}{N^{\nu_1}}a_{i_1,...,i_{\nu_0}}\bb E G_{x_1y_1}\cdots G_{x_ky_k}(G^2)_{z_1w_1}\cdots (G^2)_{z_{\nu_2}w_{\nu_2}} [G]^{\nu_4}\\=\frac{1}{-z-\bb E \ul{G}} \bb E\bigg[V_{i_1,...,i_{\nu_0}}/(G^2)_{z_1w_1}\cdot\big(G_{z_1w_1}\ul{G^2}+(G^2)_{z_1w_1}\langle \ul{G}\rangle+2N^{-1}(G^3)_{z_1w_1}\big) \\
	+\frac{1}{N}\sum_{m=1}^{k} V_{i_1,...,i_{\nu_0}}/((G^2)_{z_1w_1}G_{x_my_m}) \cdot \big((G^3)_{x_m w_1}G_{z_1y_m}+G_{x_mz_1}(G^3)_{w_1 y_m}\big)\\
	+\frac{1}{N}\sum_{m=2}^{\nu_2} V_{i_1,...,i_{\nu_0}}/((G^2)_{z_1w_1}(G^2)_{z_mw_m}) \cdot \big((G^4)_{z_m w_1}G_{z_1w_m}+(G^2)_{z_mz_1}(G^3)_{w_1w_m}\\+(G^4)_{w_1w_m}G_{z_1z_m}+(G^2)_{z_1w_m}(G^3)_{z_mw_1}\big)  \\
	+\frac{2\nu_4}{N^2}V_{i_1,...,i_{\nu_0}}/((G^2)_{z_1w_1}[G]) \cdot \big((G^4)_{z_1w_1}+2(HG^2)_{z_1w_1}mm'\big)\\
	+V_{i_1,...,i_{\nu_0}}/(G^2)_{z_1w_1} \cdot G_{z_1w_1}-K^{(1)}-\sum_i\sum_{k=2}^l L^{(1,i)}_{k} - \sum_i\cal R^{(1,z_1i)}_{l+1}\bigg]\,,
	\end{multline}
	where
	\begin{equation*} 
	K^{(1)}=\frac{a_{i_1,...,i_{\nu_0}}}{N^{1+\nu_1}}  \frac{\partial \big( G_{x_1y_1}\cdots G_{x_ky_k}(G^2)_{z_1w_1}\cdots (G^2)_{z_{\nu_2}w_{\nu_2}} [G]^{\nu_4}\big)}{\partial H_{z_1z_1}}(N\cal C_2(H_{z_1z_1})-2)\,,
	\end{equation*}
	and
	\begin{equation*}  
	L^{(1,i)}_{k}= \frac{a_{i_1,...,i_{\nu_0}}}{N^{\nu_1}}\frac{1}{k!}\mathcal{C}_{k+1}(H_{iz_1})\frac{\partial^k( G_{x_1y_1}\cdots G_{x_ky_k}(G^2)_{iw_1}(G^2)_{z_2w_2}\cdots (G^2)_{z_{\nu_2}w_{\nu_2}} [G]^{\nu_4})}{\partial H_{z_1i}^k}.  
	\end{equation*}
	Here $l$ is a fixed positive integer to be chosen later, and $\cal R_{l+1}^{(1,z_1i)}$ is a remainder term defined analogously to $\cal R_{l+1}$ in (\ref{eq:cumulant_expansion}). Again by a routine verification, the remainder term is negligible for large $l$. Note that by Theorem \ref{refthm1} we have
	\begin{equation*} 
	\frac{1}{-z-\bb E \ul{G}}=O(1)
	\end{equation*}
	uniformly for $z \in \b D_{\tau}$. Also note that
	\begin{equation} \label{4.111}
	\langle\ul{G}\rangle =[G]+O_{\prec}\Big(\frac{1}{\sqrt{N}q}\Big)=[G]+O_{\prec}\Big(\eta+\frac{1}{N\eta}\Big)\,.
	\end{equation}
	Inserting \eqref{421} into \eqref{420}, and by using \eqref{416}, \eqref{4.111} and Lemma \ref{lem4.2}, we have
	\begin{equation} \label{new}
	\begin{aligned}
	&\sum_{i_1,...,i_{\nu_0}}\bb E  V_{i_1,...,i_{\nu_0}}\\ =\,&\sum_{i_1,...,i_{\nu_0}} \frac{1}{-z-\bb E \ul{G}} \bb E\bigg[V_{i_1,...,i_{\nu_0}}/(G^2)_{z_1w_1}\cdot G_{z_1w_1}\ul{G^2}-\sum_i\sum_{k=2}^l L^{(1,i)}_k \bigg]+O_{\prec}(\cal B(V))\,.
	\end{aligned}
	\end{equation}
	
	Step 2. Now let us look closely at $L_k^{(1,i)}$. When none of the derivatives $\partial H_{z_1i}^k$ hit $[G]^{\nu_4}$, \eqref{diff} shows that all the resulting terms are still in $\cal V$. When at least one derivative hits $[G]^{\nu_4}$, we expand the factors $[G]$ that were differentiated, and split the terms according to whether $(\ul{H^2}-1)$ is differentiated or not. For example, when $k=2$, let us take the term
	\begin{multline} \label{4.31}
	\frac{a_{i_1,...,i_{\nu_0}}}{N^{\nu_1}}\frac{1}{2!}\mathcal{C}_{2+1}(H_{iz_1})G_{x_1y_1}\cdots G_{x_ky_k}(G^2)_{iw_1}(G^2)_{z_2w_2}\cdots (G^2)_{z_{\nu_2}w_{\nu_2}} [G]^{\nu_4-2} \nu_4(\nu_4-1) \Big(\frac{\partial [G]}{\partial H_{z_1i}}\Big)^2\\
	\eqd  X \Big(\frac{\partial [G]}{\partial H_{z_1i}}\Big)^2
	\end{multline}
	from $L_2^{(1,i)}$. Since 
	\begin{equation*}
	\frac{\partial [G]}{\partial H_{z_1i}}= -2N^{-1}(G^2)_{z_1i}-4N^{-1}mm'H_{z_1i} \,,
	\end{equation*}
	we can split \eqref{4.31} into
	\begin{equation} \label{4.33}
	X (-2N^{-1}(G^2)_{z_1i})^2+  X \Big(8N^{-2}mm'H_{z_1i}(G^2)_{z_1i}+16N^{-2}(mm'H_{z_1i})^2\Big)\,,
	\end{equation}
	and note that the first term in \eqref{4.33} is in $\cal V(\{G\})$, and the second term in \eqref{4.33} contains at least one derivative of $\ul{H^2}-1$. In this way, we split 
	\begin{equation} \label{split}
	L_k^{(1,i)}={L}_{k}^{(1,i,1)}+L_{k}^{(1,i,2)}
	\end{equation}
	where ${L}_{k}^{(1,i,1)}$ are all the terms in ${L}_{k}^{(1,i)}$ that do not contain the derivatives of $(\ul{H^2}-1)$. By the above reasoning, we see that ${L}_{k}^{(1,i,1)}$ is a finite linear combination of elements in $\cal V$. Also observe from \eqref{diffH} that when the derivatives hit $(\ul{H^2}-1)$, it gives us something small.
	
	By Lemma \ref{lem4.2}, \eqref{diffH} and \eqref{new}, one readily checks that
	\begin{equation} \label{4.24}
	\begin{aligned}
	&\sum_{i_1,...,i_{\nu_0}}\bb E  V_{i_1,...,i_{\nu_0}}\\ =\,&\sum_{i_1,...,i_{\nu_0}} \frac{1}{-z-\bb E \ul{G}} \bb E\bigg[V_{i_1,...,i_{\nu_0}}/(G^2)_{z_1w_1}\cdot G_{z_1w_1}\ul{G^2}-\sum_i\sum_{k=2}^l L^{(1,i,1)}_k \bigg]+O_{\prec}(\cal B(V))\,.
	\end{aligned}
	\end{equation}
	
	Step 3. Now let us handle the first term on RHS of \eqref{4.24}. Define $$ V^{(1)}_{i_1,...i_n} \deq V_{i_1,...,i_{\nu_0}}/(G^2)_{z_1w_1}\cdot G_{z_1w_1}=\frac{a_{i_1,...,i_{\nu_0}}}{N^{\nu_1}}G_{x_1y_1}\cdots G_{x_ky_k}(G^2)_{z_2w_2}\cdots (G^2)_{z_{\nu_2}w_{\nu_2}} [G]^{\nu_4} G_{z_1w_1}\,,$$
	and we look at
	\begin{equation} \label{4.25}
	\bb E V^{(1)}_{i_1,...i_n} \ul{G^2}=\frac{a_{i_1,...,i_{\nu_0}}}{N^{\nu_1}}\bb E G_{x_1y_1}\cdots G_{x_ky_k}(G^2)_{z_2w_2}\cdots (G^2)_{z_{\nu_2}w_{\nu_2}} [G]^{\nu_4} G_{z_1w_1}\ul{G^2}\,.
	\end{equation}
	Similar as in \eqref{421}, we use $z\ul{G^2}=\ul{HG^2}-\ul{G}$ and Lemma \ref{lem:cumulant_expansion} to expand \eqref{4.25}. We get
	\begin{multline} \label{425}
	\bb E V^{(1)}_{i_1,...i_n} \ul{G^2}=\frac{1}{T} \bb E\bigg[V^{(1)}_{i_1,...,i_{\nu_0}}\cdot\big(2\ul{G^2}\langle \ul{G}\rangle+2N^{-1}\ul{G^3}\big) 
	+\frac{2}{N^2}\sum_{m=1}^{k} V^{(1)}_{i_1,...,i_{\nu_0}}/G_{x_my_m} \cdot (G^4)_{x_my_m}\\
	+\frac{4}{N^2}\sum_{m=2}^{\nu_2} V^{(1)}_{i_1,...,i_{\nu_0}}/(G^2)_{z_mw_m} \cdot (G^{5})_{z_mw_m}+\frac{2\nu_4}{N^2}  V^{(1)}_{i_1,...,i_{\nu_0}}/[G] \cdot (\ul{G^4}+2\ul{HG^2}mm') \\ +V^{(1)}_{i_1,...,i_{\nu_0}} \cdot \ul{G}+\frac{2}{N^2}V^{(1)}_{i_1,...,i_{\nu_0}}/G_{z_1m_1} (G^4)_{z_my_m}
	-K^{(2)}-\sum_{i,j}\sum_{k=2}^l L^{(2,ji)}_k -\frac{1}{N}\sum_{i,j} \cal R^{(2,ji)}_{l+1}\bigg]\,,
	\end{multline}
	where $\cal R^{(2,ji)}_{l+1}$ is the remainder term,
	\begin{equation*} 
	K^{(2)}=\frac{a_{i_1,...,i_{\nu_0}}}{N^{2+\nu_1}} \sum_{i} \frac{\partial\big( G_{x_1y_1}\cdots G_{x_ky_k}(G^2)_{z_2w_2}\cdots (G^2)_{z_{\nu_2}w_{\nu_2}} [G] ^{\nu_4} G_{z_1w_1}(G^2)_{ii}\big)}{\partial H_{ii}}(N\cal C_2(H_{ii})-2)\,,
	\end{equation*}
	and
	\begin{equation*}  
	L^{(2,ji)}_k= \frac{a_{i_1,...,i_{\nu_0}}}{N^{1+\nu_1}}\frac{1}{k!}\mathcal{C}_{k+1}(H_{iz_1})\frac{\partial^k\big( G_{x_1y_1}\cdots G_{x_ky_k}(G^2)_{z_2w_2}\cdots (G^2)_{z_{\nu_2}w_{\nu_2}} [G]^{\nu_4} G_{z_1w_1}(G^2)_{ji}\big)}{\partial H^k_{ji}}.  
	\end{equation*}
	Recall that $T=-z-2\bb E \ul{G}$ satisfies \eqref{bound for T}. Similar as in \eqref{split}, we can split
	\[
	L^{(2,ji)}_k=L^{(2,ji,1)}_k+L^{(2,ji,2)}_k\,,
	\]
	where $L_{k}^{(2,ji,1)}$ is a finite linear combination of elements in $\cal V$. By inserting \eqref{425} into \eqref{4.24}, applying \eqref{diffH} and Lemma \ref{lem4.2}, we have
	\begin{equation} \label{4.399}
	\sum_{i_1,...,i_{\nu_0}} \bb E  V_{i_1,...,i_{\nu_0}} =\sum_{k=2}^{l}\sum_{i_1,...,i_{\nu_0},i,j} \frac{1}{(z+\bb E \ul{G})T} \bb EL^{(2,ji,1)}_k+\sum_{k=2}^{l}\sum_{i_1,...,i_{\nu_0},i} \frac{1}{z+\bb E \ul{G}} \bb EL^{(1,i,1)}_k+O_{\prec}(\cal B(V))\,.
	\end{equation}
	
	Step 4. Now let us see how to further (recursively) expand \eqref{4.399} and why the expansion ends in finitely many steps. Let $V$ be as in \eqref{420}. For any ${V}_*\in \cal V$ satisfying $\nu_4({V}_*)\leq \nu_4(V)$, let us define the ratio
	\begin{equation*}
	I({V}_*)\deq \frac{N^{b_0({V}_*)}}{({N\eta})^{-\nu_4(V)+\nu_4({V}_*)}N^{b_0({V})-1}}\,.
	\end{equation*}
	From Lemma  \ref{lem4.2} we know that
	\begin{equation} \label{pushu}
	\sum_{i_1,...,i_{\nu_0(V_*)}} \bb E{V}_*\prec N^{b_0(V_*)}\cal M^{\nu_4(V_*)}\leq I({V}_*)\cal B (V)\,.
	\end{equation}
	By construction, $L_k^{(1,i,1)}$ and $L^{(2,ji,1)}_k$ are finite linear combinations of the elements in $\cal V$. Let us collect these elements in the set $\cal V^{(1)}$. Pick arbitrary $V^{(1)}\in \cal V^{(1)}$. We see that $\nu_2(V^{(1)})\geq 1$, $\nu_3(V^{(1)})\geq {\nu}_3(V)$, and $\nu_4(V^{(1)})\leq \nu_4(V)$. By Lemmas \ref{Tlemh} and \ref{lem4.2} one readily check that
	\begin{equation}  \label{recursive}
	I(V^{(1)})\leq I(V) \cdot  (N\eta)^{(\nu_3(V)-\nu_3(V^{(1)}))/2}N^{-\beta}=I(V) \cdot (N\eta)^{(\widetilde{\nu}_3(V^{(1)})-\widetilde{\nu}_3(V))/2}N^{-\beta}\,,
	\end{equation}
	which together with \eqref{pushu} implies
	\[
	\sum_{i_1,...,i_{\nu_{0,1}}}  \bb E V^{(1)}\prec I(V) (N\eta)^{(\widetilde{\nu}_3(V^{(1)})-\widetilde{\nu}_3(V))/2}N^{-\beta}\cdot \cal B (V) \,.
	\]
	Here we abbreviate $\nu_{0,1}\deq\nu_0(V^{(1)})$. Repeat \eqref{4.399} we have
	\begin{equation*}
	\sum_{i_1,...,i_{\nu_{0,1}}} \bb E  V^{(1)}_{i_1,...,i_{\nu_{0,1}}} =\sum_{k=2}^{l}\sum_{i_1,...,i_{\nu_{0,1}},i,j} \frac{1}{(z+\bb E \ul{G})T} \bb E\widetilde{L}^{(2,ji,1)}_k+\sum_{k=2}^{l}\sum_{i_1,...,i_{\nu_{0,1}},i} \frac{1}{z+\bb E \ul{G}} \bb E\widetilde{L}^{(1,i,1)}_k+O_{\prec}(\cal B(V^{(1)}))\,.
	\end{equation*}
	Note that \eqref{recursive} implies
	\begin{multline*}
	\cal B(V^{(1)})\deq N^{b_0(V^{(1)})}  \cal M^{\nu_4(V^{(1)})+1}+\sum_{k=0}^{\nu_4(V^{(1)})}N^{b_0(V^{(1)})}\frac{1}{(N\eta)^k} \Big(\eta+\frac{1}{(N\eta)^{\widetilde{\nu_3}(V^{(1)})/2}}\Big)\cdot \cal M^{\nu_4(V^{(1)})-k}\\
	\leq N^{-\beta}(N\eta)^{(\widetilde{\nu}_3(V^{(1)})-\widetilde{\nu}_3(V))/2-\nu_4(V)+\nu_4(V^{(1)})}\bigg( N^{b_0(V)}  \cal M^{\nu_4(V^{(1)})+1}\\+\sum_{k=0}^{\nu_4(V^{(1)})}N^{b_0(V)}\frac{1}{(N\eta)^k} \Big(\eta+\frac{1}{(N\eta)^{\widetilde{\nu_3}(V^{(1)})/2}}\Big)\cdot \cal M^{\nu_4(V^{(1)})-k}\bigg)
	\leq N^{-\beta} \cal B(V)\,.
	\end{multline*}
	By construction, $\widetilde{L}_k^{(1,i,1)}$ and $\widetilde{L}^{(2,ji,1)}_k$ are finite linear combinations of the elements in $\cal V$. Let us collect these elements in the set $\cal V^{(2)}$. Pick arbitrary $V^{(2)}\in \cal V^{(2)}$. We see that $\nu_2(V^{(2)})\geq 1$, $\nu_3(V^{(2)}) \geq \nu_3(V^{(1)})\geq {\nu}_3(V)$ and $\nu_4(V^{(2)})\leq \nu_4(V^{(1)})\leq \nu_4(V)$. By Lemmas \ref{Tlemh}, \ref{lem4.2} and \eqref{recursive} we have 
	$$
	I(V^{(2)})\leq I(V^{(1)}) \cdot (N\eta)^{(\widetilde{\nu}_3(V^{(2)})-\widetilde{\nu}_3(V^{(1)}))/2} N^{-\beta}\leq  I(V)\cdot (N\eta)^{(\widetilde{\nu}_3(V^{(2)})-\widetilde{\nu}_3(V))/2} N^{-2\beta}\,,
	$$ which together with \eqref{pushu} implies
	\[
	\sum_{i_1,...,i_{\nu_{0,2}}} \bb E V^{(2)}\prec I(V) (N\eta)^{(\widetilde{\nu}_3(V^{(2)})-\widetilde{\nu}_3(V))/2} N^{-2\beta}\cdot \cal B (V)\,.
	\]
	Here we abbreviate $\nu_{0,2}\deq\nu_0(V^{(2)})$. Repeating the above steps we get
	\[
	\sum_{i_1,...,i_{\nu_0}} \bb E  V_{i_1,...,i_{\nu_0}} =\sum_{V^{(n)} \in \cal V^{(n)}}\sum_{i_1,...,i_{\nu_{0}(V^{(n)})}} \bb E V^{(n)}+O_{\prec}(\cal B(V))
	\]
	for any $n \in \bb N$. Here $|\cal V^{(n)}|\leq C_n$, and each $V^{(n)} \in \cal V^{(n)}$ satisfies 
	\[
	\sum_{i_1,...,i_{\nu_{0,n}}} \bb E V^{(n)}\prec I(V) (N\eta)^{(\widetilde{\nu}_3(V^{(n)})-\widetilde{\nu}_3(V))/2} N^{-n\beta}\cdot \cal B (V)\,.
	\]
	 Since $\beta>0$, $\widetilde{\nu_3}\in \{0,1,2\}$ and $I(V)= N$, setting $n$ large enough we get
	\[
	\sum_{i_1,...,i_{\nu_0}}\bb E  V_{i_1,...,i_{\nu_0}}=O_{\prec}(\cal B(V))
	\]
	as desired.

\subsection{Proof of Lemma \ref{lem4.5}} \label{sec5.2}

	Again, to simplify notations, we shall prove the lemma for $\cal Y=\{G\}$, and one easily checks that the proof is the same for $\cal Y=\{G,G^{*}\}$. Let us take a general term  $W \in \cal W^{(\nu_0,\nu_1)}(\{G\})$, and consider	
	\begin{equation}  \label{434}
	\sum_{i_1,...,i_{\nu_0}} \bb E  W_{i_1,...,i_{\nu_0}} =\frac{1}{N^{\nu_1}}\sum_{i_1,...,i_{\nu_0}}a_{i_1,...,i_{\nu_0}} \bb E \langle G_{x_1y_1}\cdots G_{x_ky_k}\rangle [G] ^{\nu_4}\,,
	\end{equation}
	where $x_1,y_1,...,x_k,y_k \in \{i_1,...,i_{\nu_0}\}$, and $a_{i_1,.,,,.i_{\nu_0}}$ are complex numbers uniformly bounded in $i_1,...,i_{\nu_0}$. Resolvent identity $zG_{x_1y_1} =(HG)_{x_1y_1}-I_{x_1y_1}$ and $\bb E \langle X \rangle Y= \bb E X \langle Y \rangle$ gives
	\[
	z\bb E \langle G_{x_1y_1}\cdots G_{x_ky_k}\rangle [G]^{\nu_4}=\bb E  (HG)_{x_1y_1}G_{x_2y_2}\cdots G_{x_ky_k} \langle[G] ^{\nu_4} \rangle-\bb E  \delta_{x_1y_1}G_{x_2y_2}\cdots G_{x_ky_k} \langle[G]^{\nu_4} \rangle\,.
	\]
	Similar as in the proof of Lemma \ref{lem4.3}, we apply Lemma \ref{lem:cumulant_expansion} to the above and get
	\begin{multline} \label{435}
	\bb E  W_{i_1,...,i_{\nu_0}}=\frac{a_{i_1,...,i_{\nu_0}}}{N^{\nu_1}}\bb E \langle G_{x_1y_1}\cdots G_{x_ky_k}\rangle [G] ^{\nu_4}\\=\frac{a_{i_1,...,i_{\nu_0}}}{-z-\bb E \ul{G}} \bb E\bigg[\frac{1}{N^{\nu_1}}\langle \delta_{x_1y_1} G_{x_2y_2}\cdots G_{x_ky_k}\rangle [G]^{\nu_4}+\frac{1}{N^{\nu_1}}\big\langle (\langle \ul{G}\rangle G_{x_1y_1}+N^{-1}(G^2)_{x_1y_1}) G_{x_2y_2}\cdots G_{x_ky_k}\big\rangle [G]^{\nu_4} \\
	+\frac{1}{N^{1+\nu_1}}\sum_{m=2}^{k} \big\langle ((G^2)_{x_1y_m}G_{y_1x_m}+(G^2)_{x_1x_m}G_{y_1y_m}) G_{x_2y_2}\cdots G_{x_{m-1}y_{m-1}} G_{x_{m+1}y_{m+1}}\cdots G_{x_ky_k} \big\rangle [G]^{\nu_4}  \\
	+\frac{2\nu_4}{N^{2+\nu_1}} G_{x_2y_2} \cdots G_{x_ky_k} ((G^3)_{y_1x_1}+2(HG)_{y_1x_1}mm')[G]^{\nu_4-1}
	-K^{(3)}-\sum_i\sum_{k=2}^l L^{(3,i)}_{k} -\sum_{i}\cal R_{l+1}^{(3,ix_1)}\bigg]\,,
	\end{multline}
	where
	\begin{equation*} 
	K^{(3)}=\frac{1}{N^{1+\nu_1}}  \bE\frac{\partial \big( G_{x_1y_1}\cdots G_{x_ky_k} \langle[G]^{\nu_4}\rangle\big)}{\partial H_{x_1x_1}}(N\cal C_2(H_{x_1x_1})-2)\,,
	\end{equation*}
	and
	\begin{equation*}  
	L^{(3,i)}_{k}= \frac{1}{N^{\nu_1}}\frac{1}{k!}\mathcal{C}_{k+1}(H_{ix_1})\bE\frac{\partial^k( \big( G_{iy_1}\cdots G_{x_ky_k} \langle[G] ^{\nu_4}\rangle)}{\partial H_{x_1i}^k}\,.  
	\end{equation*}
	Now we insert \eqref{435} into \eqref{434}, and by using \eqref{416}, \eqref{4.111} and Lemma \ref{lem4.2}, we have
	\begin{equation*} 
	\begin{aligned}
	&\sum_{i_1,...,i_{\nu_0}} \bb E  W_{i_1,...,i_{\nu_0}}\\
	=\,&\sum_{i_1,...,i_{\nu_0}} a_{i_1,...,i_{\nu_0}}\frac{1}{-z-\bb E \ul{G}} \bb E\bigg[\frac{1}{N^{\nu_1}}\langle \delta_{x_1y_1} G_{x_2y_2}\cdots G_{x_ky_k}\rangle [G]^{\nu_4}-\sum_i\sum_{k=2}^l L^{(3,i)}_{k} \bigg]+O_{\prec}\big(\cal B^{(1)}(W)\big)\,.
	\end{aligned}
	\end{equation*}
	Similar as in \eqref{split}, we split 
	\begin{equation*} 
	L_k^{(3,i)}={L}_{k}^{(3,i,1)}+L_{k}^{(3,i,2)}
	\end{equation*}
	where ${L}_{k}^{(3,i,1)}$ are all the terms in ${L}_{k}^{(3,i)}$ that do not contain the derivatives of $(\ul{H^2}-1)$. By \eqref{diffH} and Lemma \ref{lem4.2}, we see that the terms associated with $L_k^{(3,i,2)}$ are negligible, thus
	\begin{equation} \label{4.39}
	\begin{aligned}
	&\sum_{i_1,...,i_{\nu_0}} \bb E  W_{i_1,...,i_{\nu_0}}\\
	=\,&\sum_{i_1,...,i_{\nu_0}} a_{i_1,...,i_{\nu_0}}\frac{1}{-z-\bb E \ul{G}} \bb E\bigg[\frac{1}{N^{\nu_1}}\langle \delta_{x_1y_1} G_{x_2y_2}\cdots G_{x_ky_k}\rangle [G]^{\nu_4}-\sum_i\sum_{k=2}^l L^{(3,i,1)}_{k} \bigg]+O_{\prec}\big(\cal B^{(1)}(W)\big)\,.
	\end{aligned}
	\end{equation}
	Now let us look at the terms in \eqref{4.39} carefully. For any ${W_*}\in \cal W$ satisfying $\nu_4({W_*})=\nu_4(W)$, we define the ratio
	\begin{equation*} 
	I^{(1)}({W_*})\deq \frac{N^{b_1({W_*})}}{N^{b_1(W)-1}}\,,
	\end{equation*}
	where $W$ is defined as in \eqref{434}. By Lemma \ref{lem4.4}, $\sum \bb E W_* \prec I^{(1)}(W_*)\cal B^{(1)}(W)$.
	For $L^{(3,i,1)}_k$, we can apply all the differentials $\partial H_{x_1i}^k$ and write $L^{(3,i,1)}_k$ in the form of linear combinations
	\[
	L^{(3,i,1)}_k=\sum_{s=1}^{n}c_s V_s+\sum_{t=1}^{n'}c'_t W_t\,,
	\] 
	where $V_s \in \cal V$, $W_t \in \cal W$, and $c_s,c'_t$ are constants. Each $V_s$ is formed by requiring at least one differential hit $[G]^{\nu_4}$, and we can use Lemma \ref{lem4.3} to show that
	\[
	\sum_{i_1,...,i_{\nu_0},i} a_{i_1,...,i_{\nu_0}}\frac{1}{-z-\bb E \ul{G}} \bb E (-c_sV_s) =O_{\prec}(\cal B^{(1)}(W))\,.
	\] 
	Each $W_t$ is formed by requiring non of the differentials hit $[G]^{\nu_4}$, thus $\nu_4(W_t)=\nu_4(W)$. By $\cal C_{k+1}(H_{x_1i})\leq C/(Nq)$ we find that\
	$$
	I^{(1)}(W_t)\leq I^{(1)}(W) \cdot (N\eta)^{(\widetilde{\nu}_3(W_t))-\widetilde{\nu}_3(W))/2} N^{-\beta}\,,
	$$ which implies
	\begin{equation} \label{mao}
	\sum_{i_1,...,i_{\nu_0},i} a_{i_1,...,i_{\nu_0}}\frac{1}{-z-\bb E \ul{G}} \bb E (-c'_tW_t) \prec I^{(1)}(W)  (N\eta)^{(\widetilde{\nu}_3(W_t))-\widetilde{\nu}_3(W))/2} N^{-\beta}\cal B^{(1)}(W)\,.
    \end{equation}
     As for the first term on RHS of \eqref{4.39}, one easily sees that it is small enough when $x_1 \not\equiv y_1$. When $x_i\equiv y_i$ for all $i=1,...,k$, we rewrite \eqref{4.39} into
    \begin{equation} \label{w}
    \sum_{i_1,...,i_{\nu_0}} \bb E  W_{i_1,...,i_{\nu_0}}
    =\sum_{i_1,...,i_{\nu_0}} \bb E W'+\sum_{i_1,...,i_{\nu_0}} a_{i_1,...,i_{\nu_0}}\frac{1}{-z-\bb E \ul{G}} \bb E\bigg[-\sum_i\sum_{t=1}^{n'}c'_t W_t \bigg]+O_{\prec}\big(\cal B^{(1)}(W)\big)\,,
    \end{equation} where
	\[
	W'= a_{i_1,...,i_{\nu_0}}\frac{1}{-z-\bb E \ul{G}}\frac{1}{N^{\nu_1}}\langle G_{x_2y_2}\cdots G_{x_ky_k}\rangle [G] ^{\nu_4}\,,
	\]
	and each $W_t$ in \eqref{w} satisfies \eqref{mao}. Note that $I^{(1)}(W)=I^{(1)}(W')$, and $\nu_i(W)=\nu_i(W')$ for $i=0,1,3,4$, thus we can repeat \eqref{w} with $W$ replaced by $W'$. By doing this $k$ times, we have
	\begin{multline*}
	\sum_{i_1,...,i_{\nu_0}} \bb E  W_{i_1,...,i_{\nu_0}}
	=\sum_{i_1,...,i_{\nu_0}} \bb E W^{(k)}+\sum_{i_1,...,i_{\nu_0}} a_{i_1,...,i_{\nu_0}}\frac{1}{-z-\bb E \ul{G}} \bb E\bigg[-\sum_i\sum_{t=1}^{n^{(k)}}c'_t W_t \bigg]+O_{\prec}\big(\cal B^{(1)}(W)\big)\\
	=\sum_{i_1,...,i_{\nu_0}} a_{i_1,...,i_{\nu_0}}\frac{1}{-z-\bb E \ul{G}} \bb E\bigg[-\sum_i\sum_{t=1}^{n^{(k)}}c'_t W_t \bigg]+O_{\prec}\big(\cal B^{(1)}(W)\big)\,,
	\end{multline*} 
	where we used 
	\[
		W^{(k)}= a_{i_1,...,i_{\nu_0}}\frac{1}{-z-\bb E \ul{G}}\frac{1}{N^{\nu_1}}\langle 1\rangle [G]^{\nu_4}=0\,.
	\]
	To sum up, we have
\begin{equation} \label{ww}
	\sum_{i_1,...,i_{\nu_0}} \bb E  W_{i_1,...,i_{\nu_0}}
	=\sum_{i_1,...,i_{\nu_0}} a_{i_1,...,i_{\nu_0}}\frac{1}{-z-\bb E \ul{G}} \bb E\bigg[-\sum_i\sum_{t=1}^{n'}c'_t W_t \bigg]+O_{\prec}\big(\cal B^{(1)}(W)\big)
	\end{equation}
	for some $n' \in \bb N$, where each $W_t$ in \eqref{ww} satisfies \eqref{mao}. In addition, note that $I^{(1)}(W)=N$, $\widetilde{\nu}_3\in \{0,1,2\}$, and $\beta>0$.
	
	The above argument shows that similar as in Section \ref{sec5.1}, we can repeatably use \eqref{ww} finitely many times, and eventually get
	\[
	\sum_{i_1,...,i_{\nu_0}} a_{i_1,...,i_{\nu_0}} \bb E  W_{i_1,...,i_{\nu_0}}\\
	=O_{\prec}\big(\cal B^{(1)}(W)\big)
	\]
	as desired.

\appendix 
\section{Proof of Lemma \ref{prop4.4}} \label{appA}

(i) From Theorem \ref{refthm1} we have
	\begin{equation*}
	|\langle \ul{G} \rangle| \leq |\ul{G}-m|+|\bb E \ul{G}-m| \prec \frac{1}{q}+\frac{1}{N\eta}
	\end{equation*}
	uniformly for $z=E+\ii \eta \in \b S$, and this proves \eqref{Tk} for the case $m+n=1$. In addition, by an $N^{-3}$-net argument and a deterministic monotonivity result \cite[Remark 2.7, Lemma 10.2]{BK16}, we have
	\begin{equation} \label{985}
	\sup_{z \in \b S}	|\langle \ul{G}(z) \rangle| \Big(\frac{1}{q}+\frac{1}{N\eta}\Big)^{-1} \prec 1\,.
	\end{equation}
	For $m+n\geq 2$, note that 
	\begin{equation*}
	G^mG^{*n}=\eta^{-(m+n)} f\Big(\frac{H-E}{\eta}\Big)\,,
	\end{equation*}
	where 
	\[
	f(x)=\Big(\frac{x+\ii}{x^2+1}\Big)^m\Big(\frac{x-\ii}{x^2+1}\Big)^n\,.
	\]
	Writing $f_{\eta}(x)=f\big(\frac{x-E}{\eta}\big)$ and applying Lemma \ref{HS}, we have
	\begin{multline} \label{2.12}
	\langle \ul{G^{m}G^{*n}} \big\rangle = \frac{1}{\pi\eta^{m+n}} \int_{\bb C} \partial_{\bar{z}}\big(\tilde{f}_{\eta}(z)\chi\big(y/\eta\big)\big) \langle \ul{G}(z) \rangle \dd^2 z\\
	=\frac{1}{2\pi\eta^{m+n}} \int_{\bb R^2} \Big(\ii yf_{\eta}''(x)\chi(y/\eta)+\frac{\ii}{\eta} f_{\eta}(x)\chi'(y/\eta)-\frac{y}{\eta}f'_{\eta}(x)\chi'(y/\eta)\Big) \langle \ul{G}(x+\ii y) \rangle \dd x \, \dd y \,,
	\end{multline}
	where we set $\chi(y)=1$ for $|y|\leq 1$ and $\chi(y)=0$ for $|y| \geq 2$. Note that for $m+n \ge 2$, $f$ and its derivatives are in $ L^{1}(\bb R)$. The proof of \eqref{Tk} then finishes by inserting \eqref{985} into \eqref{2.12}.
	
	The proof of \eqref{410} is similar. We omit the details.
	
	(ii) The proof follows by
	\[
	\sum_{i}|(G^mG^{*n})_{ij}|^2=(G^nG^{*m}G^mG^{*n})_{jj}
	\]
	and \eqref{410}.
	
	(iii) Let us first look at the case $k=2$. By the resolvent identity $zG=HG-I$ and Lemma \ref{lem:cumulant_expansion}, we have
	\begin{equation} \label{eqn: 2.59}
	\bE \underline{G^2} = \frac{1}{T}\Big( \bE \underline{G} +2 \bE \langle \underline{G} \rangle \langle \underline{G^2} \rangle + \frac{1}{N} \bE \underline{G^3} -{\bb E K^{(4)}}- \bb E L^{(4)}-\cal R^{(4)}_{l+1}\Big)\,,
	\end{equation} 
	where $T=-z-2\bb E\ul{G}$, $\cal R^{(4)}_{l+1}$ is the remainder term,
	\begin{equation*}
	K^{(4)}=N^{-2} \sum\limits_{i} \frac{\partial (G^2)_{ii}}{\partial H_{ii}}(\bb EH_{ii}^2-2)\,,
	\end{equation*}
	and
	\begin{equation*} \label{eqn: 2.60}
	L^{(4)}= \frac{1}{N} \sum\limits_{i,j} \left( \sum\limits_{k=2}^l \frac{1}{k!} \mathcal{C}_{k+1}(H_{ji})  \frac{\partial^k (G^2)_{ij}}{\partial {H^K_{ji}}}  \right)\,.
	\end{equation*} 
The proof then follows by estimating the RHS of \eqref{eqn: 2.59} by parts (i) and (ii).

The proof of the case $k = 3$ is similar, and we omit the details.

(iv) The proof is an elementary computation. One possible way is to write
\[
\bb E (\ul{H^2}-1)^n=\frac{1}{N}\sum_{i,j}\bb E\Big(H^2_{ij}-\frac{1}{N}\Big)(\ul{H^2}-1)^{n-1}
\]
and apply Lemma \ref{lem:cumulant_expansion} with $h=H_{ij}^2$. We omit the details.

{\small
	
	\bibliography{bibliography} 
	
	\bibliographystyle{amsplain}
}
\end{document}